\documentclass[runningheads]{llncs}

\newif\ifincludeappendix
\includeappendixtrue

\usepackage[T1]{fontenc} %
\usepackage{color}

\usepackage{transparent}
\usepackage{hyperref} %
\usepackage{graphicx} %
\usepackage{cite}
\usepackage[edges]{forest} %
\usepackage{pgffor} %
\usepackage{csvsimple} %
\usepackage{underscore} %
\usepackage{listings} %
\usepackage[title]{appendix}
\usepackage{adjustbox}
\usepackage{tikz}
\usepackage{xspace}
\usepackage{wrapfig}
\usepackage{amsmath}
\usepackage{mathtools}
\usepackage{tikz}
\usepackage[ruled,linesnumbered]{algorithm2e}
\usepackage{algpseudocode}
\usepackage[acronym]{glossaries}
\usepackage{booktabs}
\usepackage{longtable}
\usepackage{pdflscape}
\usepackage{comment}
\usepackage{amssymb}
\usepackage{bbm}
\usepackage{mathrsfs}
\usepackage{blkarray}
\usepackage{pgf}
\usepackage[caption=false]{subfig}
\usepackage{rotating}
\usepackage{arydshln}
\usepackage{orcidlink}
\usepackage{fancyvrb} %
\usepackage{caption}
\usetikzlibrary{arrows.meta,positioning,fit,backgrounds} %

\newcommand{\aref}[1]{%
    \ifincludeappendix
        \hyperref[#1]{Appendix~\ref*{#1}}%
    \else
        the extended version of this paper~\cite{ref:delpi}%
    \fi
}

\SetKwInput{KwInput}{Input}
\SetKwInput{KwOutput}{Output}
\SetKwProg{InternalTry}{try}{:}{}
\SetKwProg{Catch}{catch}{:}{end}
\newcommand\Try[2]{
    \SetAlgoLined
    \InternalTry{#1}{
        \SetAlgoVlined
        #2}
    \SetAlgoVlined
}
\SetKwComment{Comment}{$\triangleright$\ }{}

\allowdisplaybreaks[1] %

\glsdisablehyper
\newglossaryentry{oop}{
    name={Object-Oriented Programming},
    first={Object-Oriented Programming (OOP)},
    text={OOP},
    description={
            Object-Oriented Programming is a programming paradigm based on the concept of objects, which can hold state in the form of fields and methods in the form of functions that can access the object's state.
        }
}
\newglossaryentry{api}{
    name = {Application Programming Interface},
    first = {Application Programming Interface (API)},
    text = {API},
    plural={APIs},
    firstplural={Application Programming Interfaces (APIs)},
    description = {
            An application programming interface (API) defines a way for two software components to communicate with each other.
        }
}
\newglossaryentry{smt}
{
    name={Satisfiability Modulo Theories},
    first={Satisfiability Modulo Theories (SMT)},
    text={SMT},
    description={
            A decision problem for logical formulas with respect to combinations of background theories
            expressed in classical first-order logic with equality.
        }
}
\newglossaryentry{fol}{
    name={First-order logic},
    first={first-order logic (FOL)},
    text={fol},
    description={
            Expressive formal system that allows the following statements:
            \begin{itemize}
                \item Constants ($x, y$)
                \item Functions ($f(x)$)
                \item Relationships ($x > y$)
                \item Connectives ($\land, \lor, \lnot, \implies, \iff$)
                \item Quantifiers ($\forall, \exists$)
            \end{itemize}
        }
}
\newglossaryentry{smtp}{
    name={SMT problem},
    text={SMT problem},
    plural={SMT problems},
    description={
            Decision problems that are true or false depending on whether a given \gls{fol} formula
            is true or false with respect to a given background theory. \\
            In the context of computer science, the background theories usually include
            arithmetic, bitvectors, arrays, etc.
        }
}
\newglossaryentry{sat}{
    first={Satisfiability (SAT)},
    name={Satisfiability},
    text={SAT},
    description={
            A formula is satisfiable if it is possible to find a set of values for its variables
            that makes the formula true.
            On the other hand, a formula is unsatisfiable if such a set does not exist.
            \begin{multline*}
                \text{Example} \\
                (a \lor c) \land (b \lor c) \land (\neg a \lor \neg c) \\
                \text{is satisfiable with the assignment } a = 1, b = 1, c = 0 \\
            \end{multline*}
        }
}
\newglossaryentry{convopt}{
    name={Convex optimization},
    text={convex optimization},
    description={
            A convex optimization problem is a problem where the objective function and the constraints
            are convex functions. \\
            A convex function is a function whose domain is a convex set and that satisfies the following
            property:
            \begin{equation*}
                f(\lambda x + (1 - \lambda)y) \leq \lambda f(x) + (1 - \lambda)f(y)
            \end{equation*}
            for all $x, y \in \text{dom}(f)$ and $\lambda \in [0, 1]$.
        }
}
\newglossaryentry{cnf}{
    first={Conjunctive Normal Form (CNF)},
    name={Conjunctive Normal Form},
    text={CNF},
    description={
            A formula is in \gls{cnf} if it is a conjunction of clauses, where a clause is a
            disjunction of literals. \\
            A literal can be a boolean variable or its negation. \\
            \\
            \begin{equation*}
                \bigwedge_{i=1}^n \bigvee_{j=1}^{m_i} l_{ij}
            \end{equation*}
        }
}
\newglossaryentry{lp}{
    first={linear programming (LP)},
    name={Linear Programming},
    text={LP},
    description={
            A Linear Programming problem is an optimization problem where the objective function and the constraints
            are linear functions. \\
            A typical Linear Programming problem has the following form:
            \begin{equation*}
                \begin{aligned}
                     & \text{maximize}   & c^T x      \\
                     & \text{subject to} & A x \leq b \\
                     &                   & x \geq 0
                \end{aligned}
            \end{equation*}
            where $x \in \mathbb{R}^d$ is the vector of variables to be determined, $c \in \mathbb{R}^d$ and $b \in \mathbb{R}^n$ are vectors of coefficients, and $A \in \mathbb{R}^{n \times d}$ is a matrix of coefficients.
        }
}
\newglossaryentry{fplp}{
    first={floating-point linear programming (FPLP)},
    name={Floating-Point Linear Programming},
    text={FPLP},
    description={
            A Floating-Point Linear Programming problem is an optimization problem where the objective function and the constraints
            are linear functions, but computations are performed in floating-point arithmetic. \\
            A typical Floating-Point Linear Programming problem has the following form:
            \begin{equation*}
                \begin{aligned}
                     & \text{maximize}   & c^T x      \\
                     & \text{subject to} & A x \leq b \\
                     &                   & x \geq 0
                \end{aligned}
            \end{equation*}
            where $x \in \mathbb{R}^d$ is the vector of variables to be determined, $c \in \mathbb{R}^d$ and $b \in \mathbb{R}^n$ are vectors of coefficients, and $A \in \mathbb{R}^{n \times d}$ is a matrix of coefficients.
        }
}
\newglossaryentry{qf-lia}{
    name={QF\_LIA},
    text={QF\_LIA},
    description={
            Quantifier-free linear integer arithmetic. In essence, Boolean combinations of inequations between linear polynomials over integer variables.
        }
}
\newglossaryentry{qf-lra}{
    first={quantifier-free linear real arithmetic (QF\_LRA)},
    name={QF\_LRA},
    text={QF\_LRA},
    description={
            Quantifier-free linear real arithmetic. In essence, Boolean combinations of inequations between linear polynomials over real variables.
        }
}
\newglossaryentry{nnf}{
    name={Negation Normal Form},
    text={NNF},
    first={Negation Normal Form (NNF)},
    description={
            A formula where the negation operator $\neg$ is only applied to variables. \\
            $\neg a$ is a NNF formula, while $\neg (a \lor b)$ is not.
        }
}
\newglossaryentry{dpll}{
    name={DPLL algorithm},
    text={DPLL},
    first={Davis-Putnam-Logemann-Loveland (DPLL)},
    description={
            A complete \gls{sat} backtracking algorithm for deciding the satisfiability of propositional logic formulas in \gls{cnf}.
        }
}
\newglossaryentry{dpllt}{
    name={DPLL(T) algorithm},
    text={DPLL(T)},
    first={Davis-Putnam-Logemann-Loveland with Theory (DPLL(T))},
    description={
            A complete \gls{sat} backtracking algorithm for deciding the satisfiability of propositional logic formulas in \gls{cnf} with theory-specific reasoning.
        }
}
\newglossaryentry{dp}{
    name={DP algorithm},
    text={DP},
    first={Davis-Putnam (DP)},
    description={
            A complete \gls{sat} algorithm that uses the resolution inference rule to decide the satisfiability of propositional logic formulas in \gls{cnf}.
        }
}
\newglossaryentry{mps}{
    name={Mathematical Programming System},
    text={MPS},
    first={Mathematical Programming System (MPS)},
    description={
            A common file format for representing \gls{lp} problems.
            Developed by IBM, it is supported by most commercial \gls{lp} solvers.
        }
}
\newglossaryentry{soi}{
    name={Sum of Infeasibilities},
    text={SoI},
    first={Sum of Infeasibilities (SoI)},
    description={
            A metric used to evaluate the quality of an \gls{lp} relaxation.
            It is the sum of the infeasibilities of the constraints that are violated by the solution.
        }
}
\newglossaryentry{onnx}{
    name={Open Neural Network Exchange},
    text={ONNX},
    first={Open Neural Network Exchange (ONNX)},
    description={
            An open-source format for representing deep learning models.
            It is supported by a wide range of tools and libraries.
        }
}
\newglossaryentry{cdlc}{
    name={Conflict Driven Learning Clause},
    text={CDLC},
    first={Conflict Driven Learning Clause (CDLC)},
    description={
            A clause that is added to the \gls{cnf} formula after a conflict is detected during the \gls{sat} solving process.
            It is used to prevent the solver from exploring the same path that led to the conflict.
        }
}
\newglossaryentry{oom}{
    name={Out Of Memory},
    text={OOM},
    first={Out Of Memory (OOM)},
    description={
            A condition where the system runs out of memory during the solving process.
        }
}
\newglossaryentry{sk}{
    name={Sloane–Stufke},
    text={SK},
    first={Sloane–Stufken (SK)},
    description={
            Set of benchmarks for testing the performance of \gls{lp} solvers on instances of linear programming problems with a large number of constraints and variables, as well as a high degree of degeneracy.
        }
}

\definecolor{keywords}{HTML}{C586C0}
\definecolor{type}{HTML}{0000FF}
\definecolor{operator}{HTML}{569CD6}
\definecolor{comments}{HTML}{6a9955}
\definecolor{variable}{HTML}{9e9b00}
\definecolor{number}{HTML}{098658}
\definecolor{function}{HTML}{795E26}

\lstdefinelanguage{SMT2}{
    alsoletter=-, %
    sensitive=true,
    morekeywords={set-logic, declare-fun, assert, check-sat, get-model},
    morekeywords=[2]{Int, Bool, Real},
    morekeywords=[3]{and, or, not, imply, ite},
    morekeywords=[4]{a, b, c, d, e, f, g, h, i, j, k, l, m, n, o, p, q, r, s, t, u, v, w, x, y, z},
    morecomment=[l]{;},
    morestring=[b]",
}
\lstdefinelanguage{mps}{
    alsoletter=-, %
    sensitive=true,
    morekeywords={NAME, ROWS, COLUMNS, RHS, RANGES, BOUNDS, BINARIES, GENERALS, ENDATA},
    morekeywords=[2]{N, L, E, G, UP, LO, FX, FR, MI, PL},
    morecomment=[l]{*},
    morestring=[b]",
}
\lstdefinelanguage{bazel}{
    alsoletter=-, %
    sensitive=true,
    morekeywords={load, workspace, http_archive, cc_binary, cc_library, cc_test, cc_toolchain, cc_toolchain_suite, filegroup, genrule, package_group, package, sh_binary, sh_library, sh_test},
    morekeywords=[2]{glob},
    morecomment=[l]{\#},
    morestring=[b]",
}
\lstdefinelanguage{yacc}{
    alsoletter={\#, \%, -, _, :}, %
    morekeywords={\#include, \#define, \#ifdef, \#endif},
    morekeywords=[2]{%
            \%union, \%token, \%type, \%nonassoc, \%left, \%prec, \%right, \%start,
            \%grammar, \%pure_parser, \%define, \%expect, \%expect-rr, \%file, \%glr-parser,
            \%parse-param, \%parse-param-rr, \%skeleton, \%debug, \%output,
            \%locations, \%skeleton
        },
    morekeywords=[3]{%
            , stringVal,
        }
    keywordstyle=\color{keywords}\bfseries,
    morecomment=[l][\color{comments}]{//}, %
    morecomment=[s][\color{comments}]{/*}{*/}, %
    morestring=[b]",
    morestring=[b]',
}
\lstdefinelanguage{flex}{
    alsoletter={\#, \%, -, _, :}, %
    morekeywords={\#include, \#define, \#ifdef, \#endif},
    morekeywords=[2]{%
            typedef
        },
    morekeywords=[3]{%
            , stringVal,
        }
    keywordstyle=\color{keywords}\bfseries,
    morecomment=[l][\color{comments}]{//}, %
    morecomment=[s][\color{comments}]{/***}{*/}, %
    morestring=[b]",
    morestring=[b]',
}
\newcounter{eqcounter}

\newcommand{\todo}[1]{}
\newcommand{\new}{}

\newcommand{\delpi}{\textit{Delpi}\xspace}
\newcommand{\pydelpi}{\textit{pydelpi}\xspace}

\newcommand{\glpk}{\textit{GLPK}\xspace}

\newcommand{\soplex}{\textit{SoPlex}\xspace}
\newcommand{\qsoptex}{\textit{QSoptex}\xspace}
\newcommand{\highs}{\textit{HiGHS}\xspace}
\newcommand{\gurobi}{\textit{Gurobi}\xspace}

\newcommand{\bg}{Bartels–Golub\xspace}
\newcommand{\ft}{Forrest–Tomlin\xspace}

\newcommand{\R}{\mathbb{R}}

\newcommand{\N}{\mathbb{N}}
\newcommand{\Q}{\mathbb{Q}}

\newcommand{\1}{\mathbbm{1}}
\newcommand{\e}{\mathbbm{e}}
\newcommand{\x}{\boldsymbol{x}}
\newcommand{\y}{\boldsymbol{y}}

\newcommand{\U}{\boldsymbol{U}}
\newcommand{\G}{\boldsymbol{G}}
\newcommand{\A}{\boldsymbol{A}}
\newcommand{\B}{\boldsymbol{B}}
\newcommand{\I}{\boldsymbol{I}}

\newcommand{\Epsilon}{\boldsymbol{\mathcal{E}}}
\newcommand{\estar}{\epsilon^\star}
\newcommand{\tstar}{\tau^\star}
\newcommand{\LI}{\boldsymbol{\mathcal{L}}}
\newcommand{\bs}[1]{\boldsymbol{#1}}
\newcommand{\hbs}[1]{\hat{\boldsymbol{#1}}}

\newcommand{\tbs}[1]{\tilde{\boldsymbol{#1}}}
\newcommand{\argmin}{\operatorname*{argmin}}
\newcommand{\fl}[2][\epsilon]{\textbf{fl}_{#1}\left(#2\right)}
\newcommand{\base}[1][]{{{\mathcal{B}}_{#1}}}

\newcommand{\maxiterations}{2^{n}}
\newcommand{\cb}{\boldsymbol{c}_{\base}}

\newcommand{\zb}{\boldsymbol{z}_{\base}}

\newcommand{\zbi}[1][i]{z_{\base_{#1}}}

\newcommand{\rmax}{r^-_{\max}}
\newcommand{\dmin}{d^+_{\min}}

\newcommand{\umin}{u^{1\sim2}_{\min}}
\newcommand{\xb}{\boldsymbol{x}_{\base}}

\newenvironment{proofsketch}
{\par\noindent\textit{Proof sketch.}\enspace}
{\par}

\begin{document}

\title{Verified Linear Programming via Tolerance-Aware Precision Boosting}
\titlerunning{Verified Linear Programming via Tolerance-Aware Precision Boosting}

\author{Ernesto Casablanca\inst{1}
    \orcidlink{0009-0009-3741-1624}
    \and
    Martin Sidaway\inst{1}
    \orcidlink{0000-0001-6481-1169}
    \and
    Sadegh Soudjani\inst{2}
    \orcidlink{0000-0003-1922-6678}
    \and
    Paolo Zuliani\inst{3}
    \orcidlink{0000-0001-6481-1169}
}
\authorrunning{E. Casablanca et al.}

\institute{Newcastle University, United Kingdom\\
\email{e.casablanca2@newcastle.ac.uk, martin.sidaway@icloud.com} \and
University of Birmingham and MPI-SWS, Birmingham, United Kingdom\\
\email{sadegh@mpi-sws.org} \and
Universit\`{a} di Roma ``La Sapienza'', Rome, Italy\\
\email{zuliani@di.uniroma1.it}}

\maketitle

\begin{abstract}
    Linear programming plays a fundamental role in computer science, with applications in optimization, formal verification, SMT solving, and numerous other domains.
    When exactness guarantees are required, numerical inaccuracies arising from floating-point arithmetic can compromise soundness, whereas exact rational arithmetic often incurs significant computational overhead.
    In this paper, we investigate the numerical stability of the simplex algorithm and establish conditions under which a precision-boosting floating-point implementation provably produces the
    same pivot decisions and final basis as an exact rational implementation, from which the final result is reconstructed and certified exactly.
    Our analysis shows that correctness depends on both arithmetic precision and a careful handling of numerical tolerances.
    Based on these results, we develop a tolerance-aware, precision-boosting simplex algorithm with formal correctness guarantees.
    Finally, we introduce a delta-complete termination criterion that allows the algorithm to terminate once certified upper and lower bounds on the optimal objective differ by at most a user-specified threshold delta, providing a certified, user-controlled optimality gap.
\end{abstract}

\section{Introduction}
\label{sec:intro}

\Gls{lp} is a fundamental tool in computer science,
with applications ranging from optimization~\cite{cite:lp-introduction} to software analysis and formal verification~\cite{cite:lp-based-software-verification,cite:lp-in-software-verification,cite:lp-in-tacas-verification},
and it is essential in any \gls{smt} solver that supports linear arithmetic theories~\cite{ref:z3-dpll-t,ref:cvc5,ref:yices}.
There exist many algorithms to solve an \gls{lp} problem, with the simplex~\cite{ref:simplex,dantzig1951maximization} and interior-point methods~\cite{cite:interior-point-revolution} being the most widely adopted.
As problems grow in size, the need for efficient and reliable solvers becomes more pressing.
Since in most settings an exact solution is not strictly required, a ``good enough'' solution is readily traded for greater computational speed.
In fact, over the years, solvers using floating-point arithmetic have achieved impressive performance;
\gurobi~\cite{ref:gurobi}, \highs~\cite{res:highs}, and \glpk~\cite{ref:glpk} are excellent examples.
However, there are cases where we need guarantees regarding the numerical error introduced by the solver.
This applies, for instance, to formal verification tools and \gls{smt} solvers, which must employ slower but exact rational arithmetic to ensure the correctness of all \gls{lp} solutions~\cite{ref:lra-dpll-t,paper:lp-for-smt}.
Here, \emph{verified} means that outcomes obtained through floating-point
computation are validated in exact rational arithmetic using
primal-dual optimality, infeasibility, or unboundedness certificates;
the correctness proofs themselves are not mechanized.
Unfortunately, in rational arithmetic, the cost of each operation is not constant and can dominate the running time of the procedure.

\paragraph{\textbf{Related Work}}
Efficient exact \gls{fplp} solvers perform most of their computations in floating-point arithmetic while relying on refinement techniques to certify the final result using exact rational arithmetic~\cite{cite:refinement-thesis,ref:soplex-precision-boosting}.
Two principal approaches have emerged.
The \qsoptex~\cite{ref:qsoptex} solver employs \textbf{incremental precision boosting}~\cite{ref:precision-boosting}, where the floating-point precision is progressively increased until a candidate solution is obtained and subsequently validated in rational arithmetic.
An alternative strategy is iterative refinement~\cite{ref:iterative-refinement}, implemented in \soplex~\cite{ref:soplex}, which typically achieves better performance but may struggle to recover from numerical difficulties \new{(e.g., severe ill-conditioning, nearly singular basis matrices, or large intermediate growth factors)}.
Although these approaches have proved highly successful in practice, the theoretical guarantees underlying precision boosting remain limited.
In particular, it has not been established under which conditions a precision-boosting simplex implementation is guaranteed to reproduce the final solution of the exact simplex algorithm.
In this paper, we address this gap by proving that precision boosting based on the \bg update can be used as a \textit{limited-precision LP oracle}~\cite{ref:lp-oracles}.

\paragraph{\textbf{Contributions}}
This paper establishes the theoretical foundations of verified linear programming through tolerance-aware precision boosting.
First, we revisit the numerical analysis of the \bg update~\cite{ref:stable-simplex}, deriving rigorous error bounds that replace the approximations in the original analysis.
\new{In particular, we establish explicit bounds for the intermediate Hessenberg matrices computed during the update process, and remove separate rounding parameters for additions and subtractions, obtaining a simpler IEEE~754-compatible formulation.}
Building on these results, we prove that there exist problem-dependent precision and tolerance parameters under which a floating-point implementation of the simplex algorithm provably follows the same pivot sequence and returns the same solution as its exact rational counterpart.
Finally, we introduce a $\delta$-complete simplex procedure that
maintains exactly validated upper and lower bounds on the optimal
objective value and terminates once their difference is at most a
user-specified threshold $\delta$.
The resulting primal-dual certificate
provides a certified, user-controlled bound on suboptimality.

\smallskip

\section{Preliminaries}
\label{sec:preliminaries}

\paragraph{\textbf{Notation}}
We will use the following notation:
$\Q$ denotes the set of rational numbers;
a bold uppercase letter ($\bs{A}, \bs{\Pi}$) denotes a matrix;
a bold lowercase letter ($\bs{a}, \bs{\pi}$) denotes a vector;
$\bs{A}_{:,j}$ ($\bs{A}_{k,:}$) denotes the $j$-th column ($k$-th row) of $\bs{A}$;
$A_{i,j}$ denotes the element in the $i$-th row and $j$-th column of $\bs{A}$;
$a_i$ denotes the $i$-th element of the vector $\bs{a}$;
$\bs{A}^T, \bs{A}^{-1}$ denote the transpose and the inverse of $\bs{A}$, respectively;
$|\bs{A}|$ denotes a matrix containing the element-wise absolute values of $\bs{A}$;
$\|\bs{A}\|_1 \coloneqq \max_{1 \le j \le n}\left(\sum_{1 \le i \le m} |A_{i,j}|\right)$ with $\bs{A} \in \Q^{m \times n}$;
$\|\bs{A}\|_{\max} \coloneqq \max_{1 \le i \le m, 1 \le j \le n}|{A_{i,j}}|$ with $\bs{A} \in \Q^{m \times n}$;
$\bs{A} \{\le, <, >, \ge \} \bs{B}$ denotes the element-wise relationship between $\bs{A}$ and $\bs{B}$;
$\bs{I}$ denotes the identity matrix;
$\LI$ denotes the unit lower triangular matrix;
$\bs{0}$ ($\bs{1}$) denotes a matrix or vector with all elements equal to $0$ ($1$), with the dimension indicated explicitly in the subscript when necessary.
This notation extends naturally to vectors.

\subsection{Floating-point Arithmetic}
\label{sec:floating-point}

Floating-point numbers belong to a set $\mathbb{F} \subset \Q$ whose elements have the form
\begin{equation*}
    y = \pm m \cdot \beta ^{e-t} ,
\end{equation*}
where $\beta \ge 2$ is the \textit{base} (or radix), $t$ is the \textit{precision} (i.e., the number of digits in the mantissa), $m$ is the integer \textit{significand} satisfying $0 \le m \le \beta^t - 1$, and $e$ is the \textit{exponent} bounded by $e_{\text{min}} \le e \le e_{\text{max}}$.
Floating-point numbers are usually stored in a normalized form, meaning $m \ge \beta^{t - 1}$ except for the number $0$.
When performing the error analysis in floating-point arithmetic, we will refer to the standard model~\cite{ref:accuracy-and-stability} since it applies to the IEEE 754 standard arithmetic~\cite{ref:ieee754} used by virtually all modern computing devices.
\begin{definition}[Standard model]
    \label{def:standard-model}
    Given two floating-point numbers $x, y \in \mathbb{F}$ and an operation $\odot \in \{+, -, \times, /\}$, the result of the operation is
    \begin{equation*}
        \fl{x \odot y} = (x \odot y)(1 + \e)^{p} \quad \text{with } |\e| \le \epsilon .
    \end{equation*}
    where $\epsilon$ is the \textit{unit roundoff} (i.e., $2^{-53} \approx 1.1 \times 10^{-16}$ for IEEE 754 double precision) and $\fl{x \odot y}$ is the floating-point computation of $x \odot y$;
    $p$ can be either $-1$ or $1$, according to convenience.
\end{definition}
Note that, by using~\autoref{def:standard-model}, we assume that neither overflow nor underflow occurs~\cite{ref:accuracy-and-stability}.
Since floating-point arithmetic is inexact, we need a way of quantifying the error introduced by the computation to judge the quality of the result.
We will consider the \textit{backward error}, which measures the smallest perturbation to the input that would yield the computed result exactly (see \autoref{fig:stability}).
\begin{figure}[h]
    \centering
    \includegraphics[alt=Diagram showing the relationship between forward and backward errors in floating-point computations.]{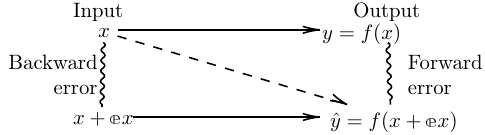}
    \caption{Forward and backward errors, adapted from~\cite{ref:accuracy-and-stability}.
        The solid arrows and dashed arrows represent an exact and a floating-point computation, respectively.}
    \label{fig:stability}
\end{figure}
We use the notation from~\cite{ref:accuracy-and-stability} to indicate error bounds over multiple operations:
\begin{lemma}
    If $|\e_i| \le \epsilon$, $p_i = \pm 1$ for $i = 1, \ldots, n$ and $n\epsilon < 1$, then
    \begin{equation*}
        \prod_{i = 1}^n(1 + \e_i)^{p_i} = 1 + \Theta_n, \qquad |\Theta_n| \le \frac{n\epsilon}{1 - n\epsilon} \eqqcolon \gamma_n .
    \end{equation*}
\end{lemma}

\subsection{Linear Programming and Simplex Algorithm}
\label{sec:linear-programming}

\gls{lp} problems are often presented in \textit{standard form},
\begin{equation}
    \label{eq:lp-standard}
    \min\left\{\bs{c}^T \x \;\middle\vert\;
    \bs{A} \x = \bs{b},\quad
    \x \geq \bs{0}\right\} ,
\end{equation}
where $\A \in \Q^{m \times n}$, $\bs{b} \in \Q^m$, $\bs{c} \in \Q^n$, and non-negative unknowns $\x \in \Q^n$.
Any \gls{lp} problem can be converted into the standard form in polynomial time~\cite{book:lp}.
As is common in the literature, we assume that $\A$ has full row rank.

An \gls{lp} problem is said to be \textit{feasible} if there exists at least one vector $\x$ (\textit{feasible solution}) that satisfies all the constraints.
If a feasible solution produces the minimum objective value (i.e., $\bs{c}^T \x$) among all feasible solutions, it is called an \textit{optimal solution}.
Every \gls{lp} in the standard form has the \textit{dual} problem
\begin{equation}
    \label{eq:lp-dual}
    \max\left\{\bs{b}^T \y \middle\vert\;
    \bs{A}^T \y + \bs{r} = \bs{c},\quad
    \bs{r} \geq \bs{0}\right\} ,
\end{equation}
with $\y \in \Q^m, \bs{r} \in \Q^{n}$.
The original problem \eqref{eq:lp-standard} is called the \textit{primal}.
It is also worth mentioning two well-established duality theorems~\cite{book:lp}.
\begin{theorem}[Weak duality]
    \label{thm:weak-duality}
    Given an \gls{lp} problem \eqref{eq:lp-standard}, let $\x$ and $\y$ be primal and dual-feasible solutions, respectively.
    Then $\bs{c}^T\x \geq \bs{b}^T\y$.
\end{theorem}
\begin{theorem}[Strong duality]
    \label{thm:strong-duality}
    Given an \gls{lp} problem \eqref{eq:lp-standard}, if the primal has an optimal solution then so does the dual.
    Moreover, they have the same value. %
\end{theorem}

The revised simplex algorithm~\cite{cite:computational-simplex} solves \gls{lp} problems by moving from one vertex of the feasible region to another while improving the objective value at each step.
Infeasibility is detected by a preliminary check over a feasible auxiliary \gls{lp} problem~\cite{cite:computational-simplex}.
At each iteration, $m$ columns of $\A$ are selected and used to create a non-singular basis matrix $\B \in \Q^{m \times m}$.
Therefore,
\begin{equation*}
    \B = \begin{bmatrix} \A_{:,\base_1} & \A_{:,\base_2} & \dots \A_{:,\base_m} \end{bmatrix} ,
\end{equation*}
where $\base_i$ is the index of the column of $\A$ that is used in the $i$-th column of the basis.
The set containing the indices of these columns is denoted by $\base$.
The remaining $n - m$ columns correspond to the \textit{non-basic variables} and form the $\bs{N} \in \Q^{m \times (n - m)}$ matrix in a similar fashion.
Setting all variables outside the basis to be $0$, we search for the so-called \textit{basic feasible solution} $\x = \B^{-1} \bs{b}, \x \ge \bs{0}$.
At each iteration, the algorithm checks whether the objective value can be improved by swapping a basic variable with a non-basic one.
The vector of \textit{reduced costs} $\bs{r} \in \Q^{n}$ is defined as
\begin{equation*}
    \y = \B^{-T}\cb, \qquad
    \bs{r} = \bs{c} - \bs{A}^T\y .
\end{equation*}
If no component of $\bs{r}$ is negative, the algorithm terminates and returns $\x$ as the optimal solution.
Otherwise, the optimal value can still be improved, and an \textit{entering variable} $x_e$ among those with negative reduced costs is selected.
The \textit{leaving variable} $x_{\base_l}$, which will be replaced in the basis by the entering variable and set to $0$, is determined by the following conditions:
\begin{equation*}
    \bs{d} = \B^{-1} \A_{:,e}, \qquad
    l = \argmin_{d_i > 0, i = 1, \ldots, m} \frac{x_{\base_i}}{d_i} .
\end{equation*}
If no component of $\bs{d}$ is positive, the \gls{lp} is unbounded, since the objective value can be improved indefinitely without violating any constraint.

\subsection{Delta-Optimality Certificates}
\label{sec:delta-complete}

To give users a principled way to trade the exact optimal solution for a faster computation, while still guaranteeing that the result is within a user-defined tolerance $\delta$ of the exact optimal value,
we introduce the concept of a $\delta$-optimal solution and $\delta$-completeness, inspired by~\cite{ref:delta-decidability}.

\begin{definition}[$\delta$-optimality certificate]
    \label{def:delta-optimality-certificate}
    Given a standard \gls{lp} problem~\eqref{eq:lp-standard} and $\delta \ge 0$, a $\delta$-optimality certificate is a pair $(\x_U , \y_L)$ such that $\x_U$ and $\y_L$ are primal and dual feasible solutions, respectively, and
    \begin{equation*}
        0 \le \bs{c}^T \x_U - \bs{b}^T \y_L \le \delta .
    \end{equation*}
\end{definition}

\begin{definition}[$\delta$-completeness]
    \label{def:delta-complete}
    Given $\delta \ge 0$, a $\delta$-complete algorithm for standard \gls{lp} problems~\eqref{eq:lp-standard} is an algorithm that terminates and outputs exactly one of the following symbols:
    \begin{itemize}
        \item $\bs{\delta}$\textbf{-optimal}: if and only if it can produce a $\delta$-optimality certificate $(\x_U , \y_L)$;
        \item \textbf{infeasible}: if and only if the \gls{lp} is infeasible;
        \item \textbf{unbounded}: if and only if the \gls{lp} is unbounded.
    \end{itemize}
\end{definition}

The \textbf{optimal} case is obtained when we solve the \gls{lp} problem exactly (i.e., $\bs{c}^T\x_U = \bs{b}^T \y_L$), and can be seen as a special case of $\delta$-optimality, with $\delta = 0$.
The soundness of a $\delta$-optimal solution is guaranteed by the duality theorems:
we are effectively using $\bs{c}^T\x_U$ and $\bs{b}^T\y_L$ as bounds around the true optimal value.

\subsection{\bg Basis Updates}
\label{sec:stable-simplex}

The revised simplex algorithm described in \autoref{sec:linear-programming} requires solving multiple linear systems at each iteration.
Instead of recomputing the inverse of the basis matrix $\B$,
most implementations exploit the structure of the problem by performing incremental updates of $\B$'s inverse~\cite{cite:product-form-simplex} or its factorization~\cite{ref:stable-simplex}.
To obtain guarantees on the error introduced by the floating-point arithmetic, we need to ensure all operations are numerically stable (i.e., the numerical errors are bounded).
The \bg update~\cite{ref:stable-simplex} is particularly promising.
\new{It is worth noting that there exist several variants to the \bg update~\cite{ref:lu-factors, ref:forrest-tomlin, res:reid} able to outperform it in terms of computational efficiency.
    In particular, the \ft~\cite{ref:forrest-tomlin} variant preserves sparsity more effectively and has lower update costs.
    However, our analysis relies on the stability guarantees obtained from the pivoting operations used in \bg.
    Since \ft and most similar variants omit these crucial steps, the corresponding backward-error analysis does not apply directly~\cite{ref:matrix-factorizations}}.

The \bg procedure starts by factorizing the basis matrix $\B$ into a permutation matrix $\bs{\Pi}$, a lower triangular matrix $\bs{L}$ and an upper triangular matrix $\U$
such that $\B = \bs{\Pi}^{-1}\bs{L}\U$.
For simplicity, we ignore the permutation matrix, since it has no impact on the algorithm's properties or its numerical stability.
All linear systems involving $\B$ are solved using \textit{forward} and \textit{backward} substitution.
When the first pivot is performed, the leaving $l_1$-th column is replaced by the $e_1$-th entering variable.
The new basis $\B^{(1)}$ becomes
\begin{equation*}
    \B^{(1)} = \begin{bmatrix} \B_{:,1} & \dots & \B_{:,l_1-1} & \B_{:,l_1+1} & \dots \B_{:,m} & \A_{:,e_1} \end{bmatrix} .
\end{equation*}
The new column is added at the end, and all other columns after the $l_1$-th are shifted left by one.
The $\bs{L}\U$ factorization needs to be updated accordingly.
The procedure keeps $\bs{L}$ fixed, therefore the left term of the factorization becomes
\begin{equation*}
    \bs{L}^{-1}\B^{(1)} = \begin{bmatrix} \U_{:,1} & \dots & \U_{:,l_1-1} & \U_{:,l_1+1} & \dots & \U_{:,m} & \bs{L}^{-1}\bs{\Pi}\A_{:,e_1} \end{bmatrix} = \bs{H}^{(1)} .
\end{equation*}
Due to the nature of matrix multiplication, the rearrangement of the columns of $\B$ can be expressed simply as the same rearrangement of the columns of $\U$.
The appended column can also be computed easily because a left-applied matrix (here, $\bs{L}^{-1}\bs{\Pi}$) modifies the columns of its matrix operand independently.
The resulting matrix $\bs{H}^{(1)}$ is a Hessenberg matrix since all columns from the $(l_1+1)$-st onward have a non-zero entry in the first sub-diagonal, followed by zeros below.
To transform $\bs{H}^{(1)}$ into an upper triangular matrix $\U^{(1)}$, we perform a Gaussian elimination procedure with partial pivoting via the application of a series of elementary permutation matrices $\bs{\Pi}^{(1)}_j$ and trivial lower triangular matrix $\bs{\Gamma}^{(1)}_j$ such that
$\bs{\Pi}^{(1)}_j$ is either the identity matrix or the identity with the $j$-th and $(j + 1)$-st rows exchanged, and $\bs{\Gamma}^{(1)}_j$ has the form
\begin{equation*}
    \bs{\Gamma}^{(1)}_j = \begin{blockarray}{cccccccc}
        1 & 2  & & j & j + 1  && m  \\
        \begin{block}{[ccccccc]c}
            1 & 0 & \dots & 0 & 0 & \dots & 0 & 1  \\
            0 & 1 & \dots & 0 & 0 & \dots & 0 & 2  \\
            \vdots &  & \ddots &  &  &  & \vdots &   \\
            0 & 0 & \dots & 1 & 0 & \dots & 0 &  j \\
            0 & 0 & \dots & -g^{(1)}_j & 1 & \dots &0 & j+1  \\
            \vdots &  &  & & & \ddots & \vdots &   \\
            0 & 0 & \dots & 0 & 0 & \dots & 1 & m  \\
        \end{block}
    \end{blockarray}
\end{equation*}
with $|g^{(1)}_j| = \frac{\min(|H_{j,j}|, |H_{j + 1,j|})}{\max(|H_{j,j}|, |H_{j + 1,j}|)} \le 1$.

The index $j$ ranges from $l_1$ to $m - 1$ because only the columns after the $(l_1-1)$-st need to be adjusted, with the exception of the last one, which does not violate the structure of a lower triangular matrix.
Hence,
\begin{equation*}
    \U^{(1)} = \bs{\Gamma}^{(1)}_{m-1}\bs{\Pi}^{(1)}_{m-1}\dots \bs{\Gamma}^{(1)}_{l_1}\bs{\Pi}^{(1)}_{l_1}\bs{H}^{(1)} .
\end{equation*}
To simplify the notation, let
\begin{equation*}
    \bs{C}^{(i)} = \bs{\Pi}^{(i)}_{l_i}\bs{\Gamma}^{(i)-1}_{l_i}\dots \bs{\Pi}^{(i)}_{m-1}\bs{\Gamma}^{(i)-1}_{m-1} ,
\end{equation*}
which is the inverse of the operator we applied to $\bs{H}^{(i)}$ (i.e., $\bs{C}^{(i)}\U^{(i)} = \bs{H}^{(i)}$), generalized to any iteration $i$.
In fact, we can repeat the same steps to update the basis matrix $\B^{(i)}$ to $\B^{(i+1)}$.
For instance, given $\B^{(1)} = \bs{L}\bs{C}^{(1)} \U^{(1)}$, we can pivot $\B^{(1)}$ to $\B^{(2)}$ while keeping $\bs{L}$ and $\bs{C}^{(1)}$ fixed,
yielding
\begin{equation*}
    \B^{(2)} = \bs{L}\bs{C}^{(1)} \bs{H}^{(2)} = \bs{L}\bs{C}^{(1)}\bs{C}^{(2)} \U^{(2)} .
\end{equation*}
Finally, letting $\G^{(i)} = \bs{C}^{(1)} \bs{C}^{(2)} \dots \bs{C}^{(i)}$, at the $i$-th iteration we have
\begin{equation*}
    \B^{(i)} = \bs{L}\bs{C}^{(1)}\dots\bs{C}^{(i)}\U^{(i)} = \bs{L}\G^{(i)}\U^{(i)} .
\end{equation*}
Having fixed the LP formulation, the simplex procedure, and the \bg update, we now turn to the numerical question underlying
our correctness result: how floating-point errors introduced by the basis
updates affect the linear systems used in the simplex decisions.
\autoref{sec:bg-error-analysis} derives the required stability bounds, which
are then used in \autoref{sec:floating-point-simplex} to establish
the decision stability of the tolerance-aware floating-point simplex
algorithm.

\section{Stability Analysis of the \bg Basis Updates}
\label{sec:bg-error-analysis}

We focus on deriving bounds on the error introduced by the \bg update.
Bartels provides an error analysis of his algorithm in~\cite{ref:stable-simplex,ref:bartels-thesis},
which we update by
(I) removing separate error factors for the terms in additions and subtractions;
(II) using a single unit roundoff value $\epsilon$ (except where the algorithm explicitly changes precision);
(III) providing a rigorous bound for $\|\bs{H}^{(i)}_{:,m}\|_1$, which is only approximated in the original paper.
The differences from the results in~\cite{ref:stable-simplex} can be identified by comparing the bounds in~\autoref{tab:bounds}.

\paragraph{Asymptotic notation.}
Throughout Sections~\ref{sec:bg-error-analysis}--\ref{sec:delta-complete-simplex}, asymptotic bounds are understood as $\epsilon \to 0$.
The \gls{lp} input data and iteration index are treated as constants.
We write $f(\epsilon)=O(g(\epsilon))$ if there exist constants
$K>0$ and $\epsilon_0>0$ such that
\begin{equation*}
    |f(\epsilon)| \leq K |g(\epsilon)|
    \qquad
    \text{for all } 0<\epsilon<\epsilon_0.
\end{equation*}
The notation is applied to scalar quantities and to matrix or vector norms.
In particular, $O(1)+O(\epsilon)=O(1)$, and finite sums and products of $O(\epsilon)$ terms are manipulated according to the usual asymptotic rules.
The hidden constants may depend on the fixed LP instance and iteration index, but not on $\epsilon$.

Using the floating-point standard model (\autoref{def:standard-model}), we study the error introduced by the update of the basis matrix at the $i$-th iteration
\begin{equation}
    \label{eq:lu-error-b}
    \B^{(i)} + \e \B^{(i)} = \bs{L}\G^{(i)}\U^{(i)} ,
\end{equation}
where $\e \B^{(i)}$ is simply the difference between the value of $\B^{(i)}$ and the true value of $\bs{L}\G^{(i)}\U^{(i)}$,
computed exactly, but using the inexact stored values of $\bs{L}$, $\G^{(i)}$ and $\U^{(i)}$.
There is no need to consider the error introduced by the matrix multiplication between $\bs{L}$, $\G^{(i)}$ and $\U^{(i)}$ since it is never actually computed.
Solving a generic linear system
\begin{equation}
    \label{eq:linear-system-error}
    (\B^{(i)} + \Epsilon^{(i)})\x = \bs{b}, \qquad \x, \bs{b} \in \Q^m
\end{equation}
using the \bg update translates to solving the systems
\begin{align}
    \label{eq:linear-systems-error-l}
    (\bs{L} + \e \bs{L})\bs{y}     & = \bs{b}   \\
    \label{eq:linear-systems-error-g}
    (\G^{(i)} + \e \G^{(i)})\bs{z} & = \bs{y}   \\
    \label{eq:linear-systems-error-u}
    (\U^{(i)} + \e \U^{(i)})\x     & = \bs{z} .
\end{align}
Substitution and expansion of the equations in \eqref{eq:linear-systems-error-l}, \eqref{eq:linear-systems-error-g}, \eqref{eq:linear-systems-error-u}, followed by substitution in \eqref{eq:lu-error-b} to eliminate the term $\bs{L}\G^{(i)}\U^{(i)}$,
and then collection of all matrix terms apart from $\B^{(i)}$ into the single combined error term $\Epsilon^{(i)}$, yields
\begin{equation}
    \begin{aligned}
        \label{eq:define-e}
        \Epsilon^{(i)} & = \e\B ^{(i)} + \bs{L}\G^{(i)}\e\U ^{(i)} + \bs{L}\e\G^{(i)}\U^{(i)} + \e\bs{L}\G^{(i)}\U^{(i)} + \bs{L}\e\G^{(i)} \e\U^{(i)} \\
                       & + \e\bs{L}\G^{(i)} \e\U^{(i)} + \e\bs{L}\e\G^{(i)}\U^{(i)} + \e\bs{L}\e\G^{(i)}\e\U^{(i)} ,
    \end{aligned}
\end{equation}
which is the a priori backward error.
\new{Since the LP instance and iteration index are fixed when taking the limit $\epsilon \to 0$, quantities such as $m, n, \|A\|_1$, and $i > 0$ are treated as constants (i.e., $O(1)$).
    Consequently, we obtain bounds such as
    \begin{equation*}
        m(1 + \epsilon) = O(1),
    \end{equation*}
    and, by Taylor expansion,
    \begin{equation*}
        (1 + \epsilon)^{m-1} - 1 = (m - 1)\epsilon + O(\epsilon^2) = O(\epsilon) .
    \end{equation*}
    The complete proofs} are given in \aref{app:proofs}, but we summarize the results in~\autoref{tab:bounds}.
\begingroup
\renewcommand{\arraystretch}{1.5}
\begin{table}[t]
    \centering
    \caption{
        Summary of the bounds on the errors introduced by the \bg update.
        Formal proofs are given in \aref{app:proofs}.
    }
    \ifincludeappendix
        \begin{tabular}{|c|c|c|c|}
            \hline
            \textbf{Reference}   & \textbf{Value}                    & \textbf{Upper Bound}                                                                                                                                                                                                        & \textbf{Big $O$ } \\
            \hline
            \eqref{eq:bound-l}   & $\|\bs{L}\|_1$                    & $m(1 + \epsilon)$                                                                                                                                                                                                           & $O(1)$            \\
            \hdashline
            \eqref{eq:bound-dl}  & $\|\e\bs{L}\|_1$                  & $\gamma_m \| \bs{L}\|_1$                                                                                                                                                                                                    & $O(\epsilon)$     \\
            \hdashline
            \eqref{eq:bound-gli} & $\|\bs{G}^{(i)-1}\bs{L}^{-1}\|_1$ & $m 2^{m + n - 1}(1 + \epsilon)^{3(m + n) - 2}$                                                                                                                                                                              & $O(1)$            \\
            \hdashline
            \eqref{eq:bound-him} & $\|\bs{H}^{(i)}_{:,m}\|_1$        & $(1 + \epsilon)^{m}\|\G^{(i)-1}\bs{L}^{-1}\|_1\|(\A_{:,e_i})\|_1$                                                                                                                                                           & $O(1)$            \\
            \hdashline
            \eqref{eq:bound-u0}  & $\|\bs{U}^{(0)}\|_1$              & $ m(2(1 + \epsilon)^3)^{m-1}\|A\|_{1}$                                                                                                                                                                                      & $O(1)$            \\
            \hdashline
            \eqref{eq:bound-ui}  & $\|\bs{U}^{(i)}\|_1$              & $m^2(1 + \epsilon)^{3(m-1)}\|\bs{H}^{(i)}\|_{1}$                                                                                                                                                                            & $O(1)$            \\
            \hdashline
            \eqref{eq:bound-dui} & $\|\e\bs{U}^{(i)}\|_1$            & $\gamma_m \|\U^{(i)}\|_1$                                                                                                                                                                                                     & $O(\epsilon)$     \\
            \hdashline
            \eqref{eq:bound-dci} & $\|\e\bs{C}^{(i)}\|_1$            & $2m((1 + \epsilon)^{m-1} - 1)$                                                                                                                                                                                              & $O(\epsilon)$     \\
            \hdashline
            \eqref{eq:bound-gi}  & $\|\G^{(i)}\|_1$                  & $m^i$                                                                                                                                                                                                                       & $O(1)$            \\
            \hdashline
            \eqref{eq:bound-dgi} & $\|\e\G^{(i)}\|_1$                & $2im^i(1 + \epsilon)^{(i-1)(m - 1)}((1 + \epsilon)^{m-1} - 1)$                                                                                                                                                              & $O(\epsilon)$     \\
            \hdashline
            \eqref{eq:bound-dai} & $\|\e\bs{A}_{e_i}\|_1$            & $\begin{aligned} (m \|\e\G^{(i)}\|_1 + m^i\|\e\bs{L}\|_1 \\ + \|\e\bs{G}^{(i)}\|_1\|\e\bs{L}\|_1) \|\bs{H}^{(i+1)}_{:,m}\|_1 \end{aligned}$                                                                                 & $O(\epsilon)$     \\
            \hdashline
            \eqref{eq:bound-db0} & $\|\e\B^{(0)}\|_1$                & $\gamma_m 2^{m-1}m^2(1 + \epsilon)^{3m - 2}\|A\|_{1}$                                                                                                                                                                       & $O(\epsilon)$     \\
            \hdashline
            \eqref{eq:bound-dbi} & $\|\e\B^{(i)}\|_1$                & $\begin{aligned} \max(\|\e\B^{(i-1)}\|_1, \|\e\bs{A}_{e_{i}}\|_1) + \\ \|\bs{L}\|_1\|\G^{(i)}\|_1\|\e\bs{C}^{(i)}\|_1\|\U^{(i)}\|_1 \end{aligned}$ & $O(\epsilon)$     \\
            \hline
        \end{tabular}
    \else
        \begin{tabular}{|c|c|c|}
            \hline
            \textbf{Value}                    & \textbf{Upper Bound}                                                                                                                                                                                                        & \textbf{Big $O$ } \\
            \hline
            $\|\bs{L}\|_1$                    & $m(1 + \epsilon)$                                                                                                                                                                                                           & $O(1)$            \\
            \hdashline
            $\|\e\bs{L}\|_1$                  & $\gamma_m \| \bs{L}\|_1$                                                                                                                                                                                                    & $O(\epsilon)$     \\
            \hdashline
            $\|\bs{G}^{(i)-1}\bs{L}^{-1}\|_1$ & $m 2^{m + n - 1}(1 + \epsilon)^{3(m + n) - 2}$                                                                                                                                                                              & $O(1)$            \\
            \hdashline
            $\|\bs{H}^{(i)}_{:,m}\|_1$        & $(1 + \epsilon)^{m}\|\G^{(i)-1}\bs{L}^{-1}\|_1\|(\A_{:,e_i})\|_1$                                                                                                                                                           & $O(1)$            \\
            \hdashline
            $\|\bs{U}^{(0)}\|_1$              & $ m(2(1 + \epsilon)^3)^{m-1}\|A\|_{1}$                                                                                                                                                                                      & $O(1)$            \\
            \hdashline
            $\|\bs{U}^{(i)}\|_1$              & $m^2(1 + \epsilon)^{3(m-1)}\|\bs{H}^{(i)}\|_{1}$                                                                                                                                                                            & $O(1)$            \\
            \hdashline
            $\|\e\bs{U}^{(i)}\|_1$            & $\gamma_m \|\U^{(i)}\|_1$                                                                                                                                                                                                     & $O(\epsilon)$     \\
            \hdashline
            $\|\e\bs{C}^{(i)}\|_1$            & $2m((1 + \epsilon)^{m-1} - 1)$                                                                                                                                                                                              & $O(\epsilon)$     \\
            \hdashline
            $\|\G^{(i)}\|_1$                  & $m^i$                                                                                                                                                                                                                       & $O(1)$            \\
            \hdashline
            $\|\e\G^{(i)}\|_1$                & $2im^i(1 + \epsilon)^{(i-1)(m - 1)}((1 + \epsilon)^{m-1} - 1)$                                                                                                                                                              & $O(\epsilon)$     \\
            \hdashline
            $\|\e\bs{A}_{e_i}\|_1$            & $\begin{aligned} (m \|\e\G^{(i)}\|_1 + m^i\|\e\bs{L}\|_1 \\ + \|\e\bs{G}^{(i)}\|_1\|\e\bs{L}\|_1) \|\bs{H}^{(i+1)}_{:,m}\|_1 \end{aligned}$                                                                                 & $O(\epsilon)$     \\
            \hdashline
            $\|\e\B^{(0)}\|_1$                & $\gamma_m 2^{m-1}m^2(1 + \epsilon)^{3m - 2}\|A\|_{1}$                                                                                                                                                                       & $O(\epsilon)$     \\
            \hdashline
            $\|\e\B^{(i)}\|_1$                & $\begin{aligned} \max(\|\e\B^{(i-1)}\|_1, \|\e\bs{A}_{e_{i}}\|_1) + \\ \|\bs{L}\|_1\|\G^{(i)}\|_1\|\e\bs{C}^{(i)}\|_1\|\U^{(i)}\|_1 \end{aligned}$ & $O(\epsilon)$     \\
            \hline
        \end{tabular}
    \fi
    \label{tab:bounds}
\end{table}
\endgroup
The last column of~\autoref{tab:bounds} expresses the bounds we are looking for in big $O$ notation, as $\epsilon \to 0$\new{: it establishes that the overall backward error satisfies $\|\Epsilon^{(i)}\|_1 = O(\epsilon)$}.
This result is crucial for proving the stability of a floating-point implementation of the simplex using the \bg update.
In fact, as long as the number of iterations $i$ is relatively small, the bound for $\Epsilon^{(i)}$ is the best we can expect, being proportional to the unit roundoff $\epsilon$.
After sufficiently many iterations, however, we are forced to re-factorize the basis to avoid catastrophic error accumulation.

\section{Tolerance-Aware Floating-Point Simplex}
\label{sec:floating-point-simplex}

Section~\ref{sec:bg-error-analysis} established that the errors introduced by the
\bg update vanish with the unit roundoff $\epsilon$. We now use
these bounds to analyze the discrete decisions of the simplex method,
including optimality, unboundedness, and pivot selection.
Since these decisions compare computed quantities against numerical thresholds,
correctness requires the tolerance $\tau$ to be chosen consistently with $\epsilon$.
\autoref{alg:stable-simplex} gives the pseudocode implementation of a floating-point primal simplex algorithm using the \bg update~\cite{ref:stable-simplex} and Bland’s pivot rule~\cite{ref:bland}.
Note that~\autoref{alg:stable-simplex} requires a feasible basis to start with, but this limitation is addressed in the $\delta$-complete algorithm.

\begin{algorithm}[h]
    \caption{Floating-point primal simplex algorithm using the \bg update~\cite{ref:stable-simplex} and Bland’s pivot rule~\cite{ref:bland}}
    \label{alg:stable-simplex}
    \SetAlgoVlined
    \KwInput{%
        $\bs{A} \in \Q^{m \times n}$, $\bs{b} \in \Q^m$, $\bs{c} \in \Q^{n}$ \newline
        Initial feasible basis indices $\base \in \N^{m}$ \newline
        Floating-point precision with unit roundoff $\epsilon \in \Q$, $\epsilon > 0$ \newline
        Tolerance $\tau \in \Q$, $\tau > \epsilon$}
    \KwOutput{Result symbol $\in \{\underline{optimal}, \underline{unbounded}\}$ \newline
        Final basis indices $\base$ and entering variable index $e \in \N$}
    $\hbs{A} \gets \fl{\bs{A}}$ \Comment*{Rounding inputs to floating-point}
    $\hbs{b} \gets \fl{\bs{b}}$ \;
    $\hbs{c} \gets \fl{\bs{c}}$ \;
    $\hbs{B} \gets \hbs{A}_{\base}$ \Comment*{Construct the basis from $\hbs{A}$ using the indices in $\base$}
    $\B \coloneqq \bs{L}\I\bs{U} \gets \text{Fact}(\hbs{B})$ \Comment*{LU factorization of $\hbs{B}$. $\B$ is an alias}
    \For{$k \gets 1$ \textbf{to} $\maxiterations$}{
        $\hbs{z}_{\base} \gets \fl{\B^{-1}\hbs{b}}$ \Comment*{Primal solution}
        $\hbs{y} \gets \fl{\B^{-T}\hbs{c}_{\base}}$ \Comment*{Dual solution}
        $\hbs{r} \gets \fl{\hbs{c} - \hbs{A}^{T}\hbs{y}}$ \Comment*{Reduced costs}
        \If{$\hbs{r} \geq -\tau \1_{n}$}{
            \label{line:stable-optimality}
            \Return \underline{optimal}, $\base$, -1 \Comment*{The value of $e$ is irrelevant}
        }
        $e \gets \min_{\hbs{r}_e < -\tau}(e)$ \Comment*{Entering variable (Bland's rule)}
        \label{line:stable-entering}
        $\hbs{d} \gets \fl{\B^{-1} \hbs{A}_{:, e}}$ \Comment*{Entering variable coefficient vector}
        \If {$\hbs{d} \le \tau \1$}{
            \label{line:stable-unbounded}
            \Return \underline{unbounded}, $\base$, $e$ \;
        }
        $\hat{u}_l \gets \begin{cases}
                \fl{\hat{z}_{\base_l} / \hat{d}_l} & \text{if } \hat{d}_l > \tau \\
                0                                  & \text{otherwise}
            \end{cases} \quad 1 \le l \le m$ \Comment*{Update candidates}
        $l \gets \min( \argmin_{\hat{d}_l > \tau, l \in \{1, 2, \dots, m\}}(\hat{u}_l))$ \Comment*{Leaving variable}
        \label{line:stable-leaving}
        $\base \gets [\base_{1}, \dots, \base_{l-1} , \base_{l+1}, \dots, \base_{m}, e]$ \;
        $\B \coloneqq \bs{L}\bs{G}\bs{U} \gets \text{BG}(\B, \hbs{A}, l)$ \Comment*{\bg update}
    }

\end{algorithm}

\subsection{Error Bounds for Simplex Computations}
\label{sec:error-bounds-simplex}
During each iteration of the algorithm, the current basis matrix $\B$ is updated as described in~\autoref{sec:stable-simplex}. %
In a real-world implementation, $\bs{L}$ and $\U$ would be re-factorized periodically, absorbing the factors from $\G$, improving the numerical stability of the algorithm.
To keep the proofs simple, this step is omitted from~\autoref{alg:stable-simplex}, since it does not affect the asymptotic bounds on the errors.

The procedure encounters four sources of error at the $i$-th iteration:
the initial rounding error introduced when storing the input parameters $\B, \bs{b}, \bs{c}$;
the update error $\e\B$;
the error $\e\U$ ($\e\bs{L}$) due to backward (or forward) substitution;
and the error $\e\G$ from applying the update factors.
Using the results proven in the following subsections, we can list a set of conditions on a positive tolerance $\tau$ that, when satisfied, ensure that \autoref{alg:stable-simplex} behaves correctly at any iteration:
\begin{enumerate}
    \item $\tau < -\rmax / 2$ and $\tau < \dmin / 2$ and $\tau < \umin$ \label{cond:tau-upper-bound}
    \item $\|\e\bs{r}\|_1 \le \tau$ and $\|\e\bs{d}\|_1 \le \tau$ and $\|\e\bs{u}\|_1 \le \min(\tau, \umin - \tau) / 2$ \label{cond:bound-errors}
\end{enumerate}
where
\begin{itemize}
    \item $\rmax$ is the maximum (i.e., least negative) value of $r_j$ across all bases and all $1 \le j \le n$ of the exact reduced cost $r_j = \bs{c}_j - (\bs{A}^T\B^{-T}\cb)_j$;
    \item $\dmin$ is the minimum positive value of $d_j$ across all bases and all $1 \le j \le m$ of the exact entering variable coefficient $d_j = (\B^{-1}\bs{A}_{:,e})_j$;
    \item $\umin$ is the smallest difference between the second-smallest and smallest distinct exact values of $u_j = \zbi[j]/d_j$ over $1 \le j \le m$ where $d_j > 0$.
\end{itemize}
Note that the values of $\rmax, \dmin$ and $\umin$ are entirely dependent on the input data $\A, \bs{b}, \bs{c}$,
but we do not compute them explicitly.
Because there are only a finite number of admissible bases, all three values are well defined.
It is clear that there exists a value $\tstar$ such that positive $\tau \le \tstar$ will satisfy condition~\ref{cond:tau-upper-bound}.
For any such $\tau$ there exists a value $\estar$ sufficiently small to ensure $\|\e\bs{r}\|_1, \|\e\bs{d}\|_1$ and $\|\e\bs{u}\|_1$ to satisfy condition~\ref{cond:bound-errors}.
We are not aware of any efficiently-computable non-trivial bounds for $\tstar$ and $\estar$ for a given problem instance.
\new{In practice, the implementation starts from a moderate tolerance $\tau$ and floating-point precision $\epsilon$, which is increased only after exact rational certification fails. 
Since the algorithm is numerically stable, often only a small number of refinements is necessary.
It is possible to construct pathological instances where the high number of refinements makes this approach less efficient than a fully rational implementation, but these are unlikely to occur in practice.}
Due to space limitations, we only provide sketched proofs for most theorems and lemmas.
The interested reader can find the rigorous proofs in \aref{app:proofs}.

To provide a rigorous bound on the error of~\autoref{alg:stable-simplex}, we start by analyzing the error of its intermediate steps.

\paragraph{\textbf{Initial Rounding Error}}
To store them in memory, the input parameters $\A, \bs{b}$ and $\bs{c}$ are converted into dyadic numbers with the appropriate precision.
After rounding, the resulting vectors $\hbs{b}, \hbs{c}$ and $\hbs{A}$ will be such that
\begin{equation}
    \label{eq:initial-rounding-error}
    \begin{aligned}
         & \hbs{A} = \fl{\bs{A}} = \bs{A} + \e \bs{A} & \qquad \|\e \bs{A}\|_1 \le \epsilon\|\bs{A}\|_1 = O(\epsilon)   \\
         & \hbs{b} = \fl{\bs{b}} = \bs{b} + \e \bs{b} & \qquad \|\e \bs{b}\|_1 \le \epsilon\|\bs{b}\|_1 = O(\epsilon)   \\
         & \hbs{c} = \fl{\bs{c}} = \bs{c} + \e \bs{c} & \qquad \|\e \bs{c}\|_1 \le \epsilon\|\bs{c}\|_1 = O(\epsilon) .
    \end{aligned}
\end{equation}
Here, the notation $\fl{\cdot}$ simply indicates the conversion of all inputs to floating-point values with unit roundoff $\epsilon$.
We can obtain the basis $\hbs{B}$ from $\hbs{A}$ by selecting a subset of its columns indexed by $\base$.
The error in $\hbs{B}$ is then given by $\e\B = \hbs{B} - \B$, where $\B$ contains the same columns, but taken from $\A$.

\begin{lemma}[Linear system error bounds]
    \label{lm:errors-linear-systems}
    Let $\hbs{b}$ be defined as in~\eqref{eq:initial-rounding-error}, $\x = \B^{-1}\bs{b}$, $\hbs{x} = \fl{\B^{-1}\bs{b}}$, where $\B$ is the basis matrix and $\hbs{x} = \x + \e\x$.
    Then
    \begin{equation*}
        \|{\e\x}\|_1 = O(\epsilon) \text{ as } \epsilon \to 0 .
    \end{equation*}
\end{lemma}
\begin{proofsketch}
    The computed solution satisfies the system $(\B+\Epsilon)(\x+\e\x)=\bs{b}+\e\bs{b}$.
    The error analysis in~\autoref{sec:bg-error-analysis} ensures that $\|\Epsilon\|_1=O(\epsilon)$ and the initial conversion gives $\|\e\bs{b}\|_1=O(\epsilon)$.
    Rearranging yields
    \begin{equation*}
        (\bs{I}+\B^{-1}\Epsilon)\e\x
        =\B^{-1}(\e\bs{b}-\Epsilon\x).
    \end{equation*}
    For sufficiently small $\epsilon$, $\bs{I}+\B^{-1}\Epsilon$ is invertible and its inverse is $O(1)$ (Neumann-series).
    The right-hand side is $O(\epsilon)$, therefore $\|\e\x\|_1=O(\epsilon)$.
\end{proofsketch}

Applying~\autoref{lm:errors-linear-systems} to the operations in~\autoref{alg:stable-simplex}, we obtain $\|\e\bs{z}_B\|_1 = O(\epsilon)$, $\|\e\bs{d}\|_1 = O(\epsilon)$ and, by replacing $\B$ with $\B^T$, $\|\e\y\|_1 = O(\epsilon)$.

\begin{lemma}[Reduced costs error bounds]
    \label{lm:errors-reduced-costs}
    Let $\hbs{c}$ be defined as in~\eqref{eq:initial-rounding-error}, $\hbs{r} = \fl{\hbs{c} - \hbs{A}^{T}\hbs{y}}$ and $ \hbs{y} = \y + \e\y$ where $\y = \B^{-T}\cb$ and $\hbs{y} = \fl{\B^{-T}\cb}$.
    If we define $\hbs{r} = \bs{r} + \e\bs{r}$ and $\bs{r} = \bs{c} - \A^{T}\y$, then
    \begin{equation*}
        \|\e\bs{r}\|_1 = O(\epsilon) \text{ as } \epsilon \to 0 .
    \end{equation*}
\end{lemma}
\begin{proofsketch}
    For $1 \le j \le n $, expand $\hat r_j=\fl{\hat c_j-\hbs{A}_{:,j}^{T}\hbs{y}}$.
    The resulting error is a sum of terms caused by rounding $\bs{c}$ and $\A$, the linear-solve error $\e\y$, and the final dot product/subtraction.
    By~\eqref{eq:initial-rounding-error} and~\autoref{lm:errors-linear-systems}, each perturbation is $O(\epsilon)$; all remaining factors are fixed for a given basis.
    Thus, $\|\e\bs{r}\|_1=O(\epsilon)$.
\end{proofsketch}

\begin{lemma}[Update vector error bounds]
    \label{lm:errors-update}
    Let $\hat{u}_j = \fl{\hat{z}_{\base_j}/\hat{d}_j}$ wherever $\hat{d}_j > 0$ and $\hat{u} = 0$ otherwise for $1 \le j \le m$.
    Similarly, let $u_j = z_{\base_j}/d_j$ wherever $d_j > 0$ and $u_j = 0$ otherwise.
    If $\hbs{z}_{\base} = \fl{\B^{-1}\hbs{b}}$ and $\hbs{d} = \fl{\B^{-1}\hbs{A}_{:,e}}$ using the same notation as in~\eqref{eq:initial-rounding-error},
    and $\hbs{u} = \bs{u} + \e\bs{u}$, then
    \begin{equation*}
        \|\e\bs{u}\|_1 = O(\epsilon) \text{ as } \epsilon \to 0 .
    \end{equation*}
\end{lemma}
\begin{proofsketch}
    For sufficiently small $\epsilon$, the sign test for $d_j$ is correct (as shown in~\aref{lm:unboundedness-determination}), therefore the rational and floating-point ratio tests consider the same indices.
    On those indices,
    \[
        \hat u_j=(1+\eta_j)\frac{z_{\base_j}+\e z_{\base_j}}{d_j+\e d_j},
        \qquad |\eta_j|\le\epsilon.
    \]
    The positive exact denominators are bounded below by the problem-dependent constant $\dmin>0$.
    Given $\|\e\zb\|_1,\|\e\bs d\|_1=O(\epsilon)$ from~\autoref{lm:errors-linear-systems}, the numerator perturbation is $O(\epsilon)$,
    while $1/(d_j+\e d_j) = O(1)$.
    Hence $\|\e\bs u\|_1=O(\epsilon)$.
\end{proofsketch}

\subsection{Termination and Correctness}

We can now prove the two fundamental properties of~\autoref{alg:stable-simplex}: termination and correctness.

\begin{lemma}[Termination]
    \label{lm:termination}
    \autoref{alg:stable-simplex} always terminates.
\end{lemma}
\begin{proof}
    Due to the use of Bland's rule~\cite{ref:bland}, the number of iterations of the algorithm is bounded by $\maxiterations$~\cite{ref:worst-case-simplex}.
    \qed
\end{proof}

\begin{theorem}[Correctness]
    \label{thm:correct-algorithm}
    Given a standard \gls{lp} problem~\eqref{eq:lp-standard} and feasible basis indices $\base$, there exists some tolerance $\tstar$ and some precision with unit roundoff $\estar$ such that,
    at any iteration,~\autoref{alg:stable-simplex} given $(\tstar, \estar)$ as input behaves the same
    as the exact version of the algorithm at all decision steps:
    (I) line~\ref{line:stable-optimality} (optimality determination) and~\ref{line:stable-entering} (entering variable choice);
    (II) line~\ref{line:stable-unbounded} (unboundedness determination);
    (III) line~\ref{line:stable-leaving} (leaving variable choice).
\end{theorem}

\begin{proofsketch}
    There are finitely many admissible bases.
    Consequently, there are finitely many exact quantities that govern the algorithm's decisions.
    In particular, let $-\rmax$, $\dmin$ and $\umin$ be positive values defined as in~\autoref{sec:error-bounds-simplex};
    if any is missing, it is replaced by $\infty$.
    Choose a tolerance $\tau^\star$ such that
    \begin{equation*}
        0 < \tau^\star <
        \min\left\{
        -\rmax/2,
        \dmin/2,
        \umin
        \right\}.
    \end{equation*}
    Lemmas \ref{lm:errors-linear-systems}--\ref{lm:errors-update}
    ensure that there exists $\epsilon^\star > 0$ sufficiently small such that
    \begin{equation*}
        \|\e\bs{r}\|_1\le\tau^\star,\qquad
        \|\e\bs{d}\|_1\le\tau^\star,\qquad
        \|\e\bs{u}\|_1<
        \min\{\tau^\star,\umin-\tau^\star\}/2,
    \end{equation*}
    thus guaranteeing that decisions (I), (II), and (III) made by~\autoref{alg:stable-simplex} coincide with those made by a rational algorithm.
\end{proofsketch}

\begin{remark}
    Although we have stated the theorem for a specific tolerance $\tstar$ and precision $\estar$,
    any $\epsilon \le \estar$ satisfies the conditions of~\autoref{thm:correct-algorithm}, as does any $\tau \le \tstar$, but not necessarily with the same value of $\estar$.
\end{remark}

\section{Certifying Delta-Complete Linear Programming}
\label{sec:delta-complete-simplex}

\autoref{thm:correct-algorithm} above establishes that
\autoref{alg:stable-simplex} reproduces the decisions of
the exact simplex method when initialized with a feasible basis and
sufficiently accurate numerical parameters $(\epsilon,\tau)$.
These
assumptions leave two issues unresolved: a feasible basis must first be constructed, and suitable values of $\epsilon$ and $\tau$ are not known a priori.
We address both issues by combining the tolerance-aware simplex procedure with a Phase-I problem~\cite{ref:simplex-phase-i} and a diagonal enumeration of increasing precisions and decreasing tolerances.
\autoref{alg:delta-complete-simplex} is the main algorithmic contribution
of the paper: it combines these ingredients into an end-to-end
certifying procedure for arbitrary rational LP instances.

For each parameter pair, \autoref{alg:infeasibility-check-simplex} applies the floating-point simplex procedure to the Phase-I problem and validates the result in exact rational arithmetic.
A certified positive optimum proves that the original LP is infeasible, whereas a certified zero optimum provides a feasible basis after the artificial variables have been removed.
\autoref{alg:delta-complete-simplex} then applies the same procedure to the original LP.
Exactly validated primal- and dual-feasible solutions provide upper and lower bounds on the optimal value, respectively; their associated witnesses form a $\delta$-optimality certificate once the gap is at most $\delta$.
Exact validation similarly establishes unboundedness when an improving feasible ray is found.
\autoref{thm:delta-complete-termination} and \ref{thm:correct-delta-complete} establish termination and correctness.

\setlength{\textfloatsep}{5pt}
\begin{algorithm}[h]
    \caption{\gls{lp} infeasibility check}
    \label{alg:infeasibility-check-simplex}
    \SetAlgoVlined
    \KwInput{$\bs{A} \in \mathbb{Q}^{m \times n}, \bs{b} \in \mathbb{Q}^m, \bs{c} \in \mathbb{Q}^{n}$, feasible basis indices $\base \in \N^{m}$ \newline
        Floating-point precision with unit roundoff $\epsilon \in \Q$, $\epsilon > 0$ \newline
        Tolerance $\tau \in \Q$, $\tau > \epsilon$}
    \KwOutput{Result symbol $\in \{\underline{\text{feasible}}, \underline{\text{infeasible}}, \underline{\text{error}}\}$ \newline
        Last basis index assignment $\base$}
    (result, $\base$, $e$) $\gets$ \autoref{alg:stable-simplex}($\A, \bs{b}, \bs{c}$, $\base$, $\epsilon$, $\tau$)\;
    \label{line:infeasibility-check-simplex:call}
    \Try{}{
        Factorize the basis $\B \coloneqq \A_{\base}$ in rational arithmetic, yielding $\B^{-1}$\;
    }
    \Catch{Basis is non-invertible}{
        \Return \underline{error}, $\base$ \Comment*{Numerical difficulties}
    }
    $\zb \gets \B^{-1} \bs{b}$ \;
    $\omega \gets \cb^T \zb$ \;
    $\y \gets \B^{-T} \bs{c}_{\base}$ \;
    $\bs{r} \gets \bs{c} - \A^T \y$ \;
    \If{$\zb \ge 0$ and $\bs{r} \ge 0$ and $\omega > 0$} {
        \label{line:infeasibility-check-simplex:infeasible}
        \Return \underline{infeasible}, $\base$ \;
    }
    \ElseIf{$\zb \ge 0$ and $\bs{r} \ge 0$ and $\omega = 0$} {
        \label{line:infeasibility-check-simplex:feasible}
        \Return \underline{feasible}, $\base$ \;
    }
    \Return \underline{error}, $\base$ \Comment*{\autoref{alg:stable-simplex} did not yield a valid solution}
\end{algorithm}

\begin{algorithm}[t!]
    \caption{$\delta$-complete LP simplex algorithm}
    \label{alg:delta-complete-simplex}
    \SetAlgoVlined
    \KwInput{$\bs{A} \in \mathbb{Q}^{m \times n}, \bs{b} \in \mathbb{Q}^m, \bs{c} \in \mathbb{Q}^{n}, \delta \in \mathbb{Q}_{\ge 0}$}
    \KwOutput{Result symbol $\in \{\underline{\delta\text{-optimal}}, \underline{\text{unbounded}}, \underline{\text{infeasible}}\}$ \newline
        $\delta$-certificate $(\omega_L, \omega_U)$ if applicable, $\bot$ otherwise}
    \KwData{Infinite sequence of precisions with unit roundoffs $E = \{\epsilon_0, \epsilon_1, \dots\} \subseteq \Q$
        and tolerances $T = \{\tau_0 , \tau_1, \dots\} \subseteq \Q$ converging to zero}

    $\A_F \gets \begin{bmatrix} \A & \I_m & -\I_m \end{bmatrix}$ \Comment*{Feasibility coefficients}
    \label{line:delta-complete-simplex:feasibility-lp-start}
    $\bs{c}_{F_i} \gets \begin{bmatrix} \bs{0}_n^T, \bs{1}_{2m}^T \end{bmatrix}^T$ \Comment*{Feasibility objective}
    $\base[F]_i \gets  \begin{cases} n + i & \bs{b}_i \ge 0 \\ n + m + i & \text{otherwise}\end{cases} \forall i \in \{1, \dots, m\}$
    \Comment*{Feasible basis indices}
    \label{line:delta-complete-simplex:feasibility-lp-end}

    $(\omega_L, \omega_U) \gets (-\infty, \infty)$ \;
    \For{$k \gets 0 \text{ to } \infty$}{
        \For{all $(\epsilon_i , \tau_j ) \in \Q \times \Q$ such that $i + j = k$ and $\epsilon_i < \tau_j$} {

            (result, $\base$) $\gets$ \autoref{alg:infeasibility-check-simplex}($\A_F, \bs{b}, \bs{c}_F, \base[F], \epsilon_i, \tau_j$) \;
            \label{line:delta-complete-simplex:infeasibility-check}
            \uIf{result = \underline{infeasible}}{
                \Return \underline{infeasible}, $\bot$ \;
                \label{line:delta-complete-simplex:infeasible}
            }
            \ElseIf{result = \underline{error}}{
                Go to the next precision \;
            }
            \If{$\base$ contains indices greater than $n$} {
                Pivot out the auxiliary variables if possible; otherwise the corresponding constraint is redundant and may be deleted~\cite{ref:simplex-phase-i}.
            }

            (result, $\base$, $e$) $\gets$ \autoref{alg:stable-simplex}($\bs{A}, \bs{b}, \bs{c}, \base, \epsilon_i, \tau_j$) \;
            \Try{}{
                factorize the basis $\B \coloneqq \A_{\base}$ in rational arithmetic, yielding $\B^{-1}$\;
            }
            \Catch{Basis is non-invertible} {
                Go to the next precision \;
            }
            $\xb \gets \B^{-1} \bs{b}$ \;
            \If{$\xb \ge 0$}{
                $\omega_{U} \gets \min(\bs{c}^T_{\base}\xb, \omega_U)$ \Comment*{Primal feasible; Upper bound}
            }
            $\y \gets \B^{-T}\cb$ \;
            $\bs{r} \gets \bs{c} - \A^T \y$ \;
            \uIf{$\bs{r} \ge 0$}{
                $\omega_{L} \gets \max(\bs{b}^T \y, \omega_L)$ \Comment*{Dual feasible; Lower bound}
            } \ElseIf{$r_e < 0$ and $\xb \ge 0$ and result = \underline{unbounded}}{
                $\bs{d} \gets \B^{-1}\A_{:,e}$ \;
                \If{$\bs{d} \le 0$}{
                    \Return \underline{unbounded}, $\bot$ \Comment*{The original LP is unbounded}
                    \label{line:delta-complete-simplex:unbounded}
                }
            }
            \If{$\omega_{U},\omega_{L}$ are finite values and $\omega_{U} - \omega_{L} \le \delta$}{
                \Return \underline{$\delta$-optimal}, $(\omega_{U}, \omega_{L})$ \;
                \label{line:delta-complete-simplex:delta}
            }
        }
    }
\end{algorithm}

\begin{theorem}[Termination]
    \label{thm:delta-complete-termination}
    \autoref{alg:delta-complete-simplex} always terminates in finite time.
\end{theorem}
\begin{proof}
    From~\autoref{lm:termination} and \autoref{thm:correct-algorithm}, we know that~\autoref{alg:stable-simplex} always terminates and produces the exact result given sufficiently small tolerance $\tau^\star$ and unit roundoff $\epsilon^\star$.
    Since the sequences $T$ (of tolerances) and $E$ (of unit roundoffs) in \autoref{alg:delta-complete-simplex} converge to zero, they contain both $\tau_{j^\star} \le \tau^\star$ and $\epsilon_{i^\star} \le \epsilon^\dagger$, where $\epsilon^\dagger \le \epsilon^\star$ is the unit roundoff corresponding to $\tau_{j^\star}$ as explained in~\autoref{sec:error-bounds-simplex}.
    Diagonal enumeration processes $(\epsilon_{i^\star}, \tau_{j^\star})$ after finitely many iterations.
    \qed
\end{proof}

\begin{theorem}[Correctness]
    \label{thm:correct-delta-complete}
    \autoref{alg:delta-complete-simplex} always returns correct results according to the definition of $\delta$-completeness in \autoref{def:delta-complete}.
\end{theorem}
\begin{proofsketch}
    Each return is validated by exact rational arithmetic.
    An \underline{infeasible} result is issued only when the auxiliary feasibility LP has a strictly positive certified optimum, which means that the original constraints cannot be satisfied.
    An \underline{unbounded} result requires an exactly primal-feasible basis, a negative reduced cost, and an exactly nonpositive direction $\B^{-1}\A_{:,e}$; the entering variable can therefore increase without limit while preserving feasibility and improving the objective.
    Finally, a \underline{$\delta$-optimal} result is returned only after finding an exactly primal-feasible solution of value $\omega_U$ and an exactly dual-feasible solution of value $\omega_L$ with $\omega_U-\omega_L\le\delta$,
    thereby yielding a valid $\delta$-optimality certificate.
\end{proofsketch}

\section{Experimental Evaluation}
\label{sec:benchmarks}

In this section, we show that our approach correctly solves a challenging class of \gls{lp} problems on which standard floating-point solvers often fail, and does so faster than other approaches.

\noindent{\textbf{Implementation and experiment setup.}}
We extended the precision-boosting \qsoptex solver by adding a $\delta$-complete mode, and integrated its routines into a modern framework that we call \delpi.\footnote{\url{https://github.com/TendTo/delpi}}
\delpi can also be built with the state-of-the-art \soplex solver, allowing users to choose between the two underlying solvers.
For ease of use, it can also be built as a Python package, \pydelpi.
We compare \delpi's running times with those of the standalone \soplex solver.

All experiments were conducted on a RHEL 11.4.1 HPC platform with an AMD EPYC 9745 @ 2.4 GHz processor.
Each run is restricted to 4 GB of memory, 1 core, and a timeout of 6 hours.
The \soplex solver (version \texttt{8.0.3}) is used with two configurations:
the default, hybrid configuration (iterative refinement and precision boosting), %
and the pure precision boosting configuration. %
\footnote{We use the \texttt{exact} and \texttt{exact-pure-boosting} settings from \url{https://github.com/scipopt/soplex/tree/master/settings}}
\new{To ensure a fair comparison, we set $\delta=0$ for all runs, which means that \delpi will return an exact optimal solution.}

\noindent{\textbf{Benchmark instances.}}
We evaluate the performance of our algorithm on the \gls{sk}~\cite{ref:sloane}  \gls{lp} benchmarks, which have been used for benchmarking in~\cite{ref:precision-boosting}.
The \gls{sk} \gls{lp}s are dense and feasible, and are characterized by five parameters ($s_1, s_2, k_1, k_2, t$), which rapidly increase the problem size and numerical complexity.
Their challenging numerical structure makes them difficult for state-of-the-art floating-point solvers.
In fact, \gurobi~\cite{ref:gurobi} (version \texttt{12.0.3}) incorrectly declares all instances as infeasible and \highs~\cite{res:highs} (version 1.15.1) cannot handle the large coefficients in the input.
\new{We also run the same experiments with \glpk~\cite{ref:glpk} (version \texttt{5.0}) in exact mode (i.e., enabling its fully rational implementation with the \texttt{--exact} flag), but it fails to solve any of the instances within the time limit.}
We report results for the 69 instances generated by varying $k_1,k_2\in[18,28]$, $s_1,s_2\in[35,50]$, and $t\in[30,40]$ and on which at least one solver reports a solution.

{
\setlength{\tabcolsep}{6pt}

\begin{table}[h]
    \caption{Summary of the results.
        The rows indicate the number of instances solved optimally 
        and the number of instances for which the solver was the fastest among those that found an optimal solution.
        \delpi uses our modified version of \qsoptex.
        The subscript indicates the solver configuration: \textit{Hy} for the hybrid (precision boosting and iterative refinement), and \textit{PB} for pure precision boosting.
    }
\centering
    \begin{tabular}{lccc}
\toprule
Solver & \delpi & \soplex\textsubscript{Hy} & \soplex\textsubscript{PB} \\
\midrule
Solved & 69 & 47 & 31 \\
Fastest & 67 & 3 & 3 \\
\bottomrule
\end{tabular}

\label{tab:solver-summary}
\end{table}

}

\noindent{\textbf{Results.}}
As summarized in~\autoref{tab:solver-summary}, \delpi is particularly effective on this family of problems and solves most instances faster than other solvers.
Forcing \soplex to use only precision boosting improves its performance, but it does not change the outcome.
Note that the best heuristic is often problem-dependent: iterative refinement can be the better choice in other cases~\cite{ref:iterative-refinement}.
Additional details and runtimes are provided in~\aref{app:benchmarks}.

\section{Conclusion}
\label{sec:conclusion}

In this paper, we established numerical foundations for reliable
\gls{lp} reasoning in formal verification workflows.
Focusing on the
\bg update, we derived a priori bounds on its backward-error propagation
and used them to prove that, for a suitable problem-dependent precision
and tolerance pair, the floating-point procedure reproduces the decisions
of its exact rational counterpart.
Building on these results, we
developed a $\delta$-complete simplex procedure for rational \gls{lp}
problems in standard form.
The procedure returns either a
$\delta$-optimality certificate with an optimality gap of at most a
user-specified threshold $\delta$, or a certified result of
infeasibility or unboundedness.
These guarantees are particularly
relevant to \gls{smt} solvers for linear arithmetic, where an incorrect
numerical decision can compromise the soundness of the overall
verification result.
Our experimental evaluation on the numerically challenging
Sloane--Stufken benchmark family demonstrated the practical potential of
the approach.
Future work will include the exploitation of $\delta$-complete algorithm more extensively, especially in the context of \gls{smt} solving.

 \section*{Acknowledgements}
 The authors thank Ambros Gleixner for his assistance with the use of \soplex.
 The work of Ernesto Casablanca was supported by the Engineering and Physical Sciences Research Council (EPSRC), grant number EP/W524700/1.
 The work of Sadegh Soudjani was supported by the EIC grants 101070802 and 101306368, and the ERC grant 101089047. The work of Paolo Zuliani was supported by the Sapienza University of Rome grant QUMOBA (RM125199BAE72522).

\paragraph{\textbf{Data-Availability Statement}}

The source code of all the \gls{lp} solvers used in this work, with the exception of the commercial solver \gurobi, is publicly available:
\begin{itemize}
    \item \delpi: \url{https://github.com/TendTo/delpi}
    \item \qsoptex: \url{https://github.com/TendTo/qsopt-ex}
    \item \soplex: \url{https://github.com/scipopt/soplex}
    \item \glpk: \url{https://www.gnu.org/software/glpk/}
    \item \highs: \url{https://github.com/ERGO-Code/HiGHS}
\end{itemize}
The Sloane--Stufken benchmarks can be generated by providing the desired parameters to the \texttt{eg_sloan} script that will be automatically compiled when building \qsoptex from source.
The complete list of parameters used in each instance of the experiments is available in~\aref{app:benchmarks}.

\bibliographystyle{splncs04}
\bibliography{resources}

\ifincludeappendix
    \newpage
    \appendix

\section{Experimental Results}
\label{app:benchmarks}

Additional details on the experimental results are provided in this section.
We present the trends observed in the performance of \delpi and \soplex in~\autoref{fig:results-sk}, using performance profiles~\cite{cite:performance-profiles}.
Performance profile plots show, on the vertical axis, the percentage of problems for which the solver finishes within a given tolerance from the fastest solver, shown on the horizontal axis; this provides a meaningful high-level comparison across large benchmark sets.

\begin{figure}[h!]
    \centering
    \includegraphics[alt=Performance profile plot highlighting the performance of dlinear against soplex in both of its configurations]{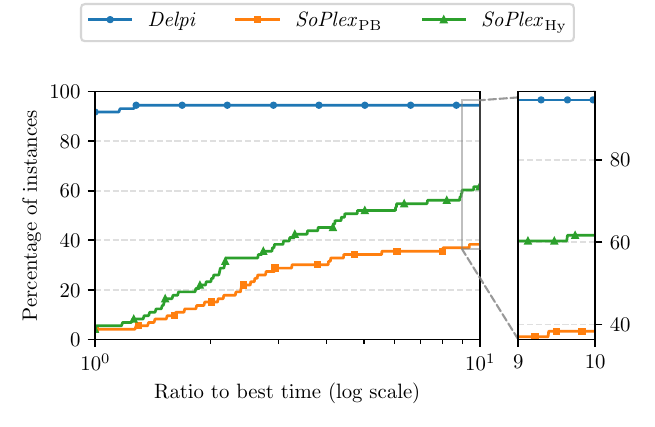}
    \caption{Performance profiles obtained for the \gls{sk} benchmarks, with a time limit of 6 hours.
        The subscript indicates the solver configuration: \textit{Hy} for the hybrid (precision boosting and iterative refinement), and \textit{PB} for pure precision boosting.
        We notice how \delpi is the fastest solver, and it is able to solver more instances than \soplex in both configurations.
    }
    \label{fig:results-sk}
\end{figure}

\newpage
{
    \setlength{\tabcolsep}{6pt}
    \begin{longtable}{cccccccc}
\caption{The timings obtained by all solvers on all 69 \gls{sk} results, sorted by the parameters $s_1$, $s_2$, $k_1$, $k_2$ and $t$.
\delpi uses our modified version of \qsoptex.
The pedices indicate the solver configuration: \textit{Hy} for the hybrid (precision boosting and iterative refinement), and \textit{PB} for pure precision boosting.} \label{tab:solver-all} \\
\toprule
s1 & s2 & k1 & k2 & t & \textit{Delpi} & \textit{SoPlex}\textsubscript{Hy} & \textit{SoPlex}\textsubscript{PB} \\
\midrule
\endfirsthead
\caption[]{The timings obtained by all solvers on all 69 \gls{sk} results, sorted by the parameters $s_1$, $s_2$, $k_1$, $k_2$ and $t$.
\delpi uses our modified version of \qsoptex.
The pedices indicate the solver configuration: \textit{Hy} for the hybrid (precision boosting and iterative refinement), and \textit{PB} for pure precision boosting.} \\
\toprule
s1 & s2 & k1 & k2 & t & \textit{Delpi} & \textit{SoPlex}\textsubscript{Hy} & \textit{SoPlex}\textsubscript{PB} \\
\midrule
\endhead
\midrule
\multicolumn{8}{r}{Continued on next page} \\
\midrule
\endfoot
\bottomrule
\endlastfoot
5 & 7 & 18 & 18 & 35 & 112.74 & 541.77 & 627.13 \\
5 & 7 & 19 & 19 & 38 & 288.71 & 784.55 & 936.84 \\
8 & 9 & 18 & 18 & 35 & 151.67 & 1460.76 & 4510.36 \\
11 & 12 & 17 & 19 & 35 & 590.40 & 1286.20 & 1532.01 \\
35 & 35 & 18 & 19 & 31 & 998.44 & 6030.51 & / \\
35 & 35 & 18 & 19 & 37 & 58.15 & 516.02 & 642.91 \\
35 & 35 & 18 & 19 & 38 & 148.28 & 660.71 & 598.65 \\
35 & 35 & 18 & 19 & 39 & 147.70 & 620.68 & 605.10 \\
35 & 35 & 18 & 21 & 38 & 1079.80 & 1443.55 & / \\
35 & 35 & 18 & 23 & 32 & 2462.47 & 7191.71 & / \\
35 & 35 & 18 & 24 & 32 & 2662.25 & 11063.91 & / \\
35 & 35 & 18 & 24 & 36 & 5804.09 & 4582.04 & 7825.37 \\
35 & 35 & 18 & 26 & 32 & 7975.13 & / & / \\
35 & 35 & 18 & 26 & 33 & 8413.49 & / & / \\
35 & 35 & 19 & 21 & 39 & 2343.43 & 3080.80 & 2026.84 \\
35 & 35 & 19 & 26 & 30 & 2784.06 & / & / \\
35 & 35 & 19 & 27 & 38 & / & / & / \\
35 & 35 & 20 & 23 & 33 & 3932.81 & 7984.49 & / \\
35 & 35 & 20 & 24 & 31 & 4693.02 & / & / \\
35 & 35 & 21 & 20 & 32 & 1482.80 & 10833.67 & / \\
35 & 35 & 21 & 21 & 32 & 7113.99 & 10247.32 & / \\
35 & 35 & 22 & 19 & 34 & 1361.29 & 5944.89 & / \\
35 & 35 & 22 & 21 & 30 & 2228.54 & 13534.53 & / \\
35 & 35 & 22 & 22 & 30 & 9766.71 & / & / \\
35 & 35 & 22 & 22 & 31 & 7219.75 & 11863.24 & / \\
35 & 35 & 22 & 23 & 31 & 7580.13 & 16424.72 & / \\
35 & 35 & 22 & 24 & 31 & 11863.86 & / & / \\
35 & 35 & 22 & 25 & 30 & 3887.44 & / & / \\
35 & 35 & 22 & 27 & 31 & 14291.66 & / & / \\
35 & 35 & 23 & 18 & 30 & 1534.58 & 13717.19 & / \\
35 & 35 & 23 & 22 & 35 & 4750.93 & / & / \\
35 & 35 & 24 & 19 & 32 & 2780.03 & / & / \\
35 & 35 & 24 & 20 & 32 & / & / & 18398.80 \\
35 & 35 & 24 & 23 & 30 & 4080.79 & / & / \\
35 & 35 & 24 & 24 & 38 & 6482.03 & / & / \\
35 & 35 & 25 & 19 & 32 & 7666.16 & / & / \\
35 & 35 & 26 & 18 & 31 & 5473.18 & / & / \\
35 & 35 & 26 & 20 & 30 & 3288.14 & / & / \\
35 & 35 & 26 & 21 & 30 & 13582.00 & / & / \\
35 & 35 & 26 & 21 & 33 & 11307.27 & / & / \\
35 & 35 & 27 & 18 & 32 & 4720.05 & / & / \\
35 & 35 & 27 & 19 & 31 & 2412.46 & / & / \\
35 & 35 & 27 & 19 & 38 & 4029.40 & / & / \\
35 & 35 & 27 & 21 & 31 & 10565.93 & / & / \\
35 & 36 & 19 & 18 & 34 & 1704.31 & 4527.94 & / \\
35 & 36 & 19 & 19 & 34 & 2229.99 & 4176.18 & 3433.96 \\
35 & 36 & 20 & 19 & 30 & 2035.55 & 7731.09 & 3723.21 \\
35 & 36 & 21 & 18 & 37 & / & 1432.35 & 2047.95 \\
35 & 36 & 21 & 19 & 32 & 3347.52 & 6486.82 & 8035.36 \\
35 & 36 & 21 & 20 & 32 & 1759.63 & 5076.98 & 14097.84 \\
35 & 36 & 21 & 21 & 39 & 1444.62 & 4764.57 & 3821.60 \\
35 & 36 & 24 & 18 & 35 & 4288.65 & 5925.48 & / \\
35 & 36 & 24 & 20 & 37 & 6759.10 & / & / \\
35 & 36 & 25 & 18 & 31 & 4286.06 & 6474.20 & 18962.84 \\
35 & 36 & 25 & 19 & 39 & / & 7117.87 & / \\
35 & 36 & 26 & 18 & 36 & / & / & / \\
35 & 36 & 26 & 18 & 37 & 2231.85 & 20061.60 & / \\
35 & 36 & 27 & 18 & 38 & 5495.93 & / & / \\
35 & 37 & 18 & 18 & 35 & 658.96 & 1200.47 & 1931.25 \\
35 & 37 & 18 & 18 & 36 & 225.96 & 336.00 & 364.81 \\
35 & 37 & 18 & 18 & 37 & 176.44 & 354.11 & 422.18 \\
35 & 37 & 18 & 18 & 38 & 171.59 & 529.64 & 434.56 \\
35 & 37 & 18 & 18 & 39 & 191.21 & 612.62 & 444.10 \\
35 & 37 & 20 & 18 & 38 & 105.20 & 1129.12 & 988.32 \\
35 & 37 & 20 & 18 & 39 & 98.33 & 1054.47 & 1245.85 \\
35 & 37 & 22 & 18 & 34 & 3829.03 & 4484.66 & / \\
35 & 37 & 23 & 19 & 38 & 3071.63 & 6724.02 & 5895.82 \\
35 & 37 & 27 & 18 & 34 & 6135.30 & / & / \\
35 & 37 & 27 & 18 & 36 & 7409.90 & / & / \\
35 & 38 & 19 & 18 & 36 & 1852.14 & 1861.90 & 2367.91 \\
35 & 38 & 19 & 18 & 39 & 457.00 & 725.25 & 985.13 \\
40 & 80 & 17 & 17 & 33 & / & 670.66 & 537.32 \\
40 & 80 & 18 & 18 & 33 & 777.48 & 1633.04 & 1622.99 \\
50 & 70 & 18 & 18 & 33 & 651.30 & 1375.24 & 1809.27 \\
50 & 80 & 18 & 18 & 32 & 1139.95 & 4065.56 & 1561.67 \\
\end{longtable}

}

\section{Proofs}
\label{app:proofs}
We have collected a series of proofs of some of the claims we made in the main text.

\subsection{Product-Difference Identity}
\label{app:matrix-expression}
Given two series of matrices $\bs{M}_1, \bs{M}_2, \dots, \bs{M}_p$ and $\bs{N}_1, \bs{N}_2, \dots, \bs{N}_p$,
we want to derive the identity
\begin{equation*}
    \begin{aligned}
         & \bs{M}_1\bs{M}_2\dots \bs{M}_p - \bs{N}_1\bs{N}_2\dots \bs{N}_p                                    \\
         & = (\bs{M_1} - \bs{N}_1)\bs{N}_2\dots \bs{N}_p + \bs{M}_1(\bs{M}_2 - \bs{N}_2)\bs{N}_3\dots\bs{N}_p \\
         & + \dots + \bs{M}_1\bs{M}_2\dots\bs{M}_{p-1}(\bs{M}_p - \bs{N}_p) .
    \end{aligned}
\end{equation*}
The case $p = 1$ is trivial, and the case $p = 2$ is also easy to verify by expanding the right-hand side:
\begin{equation*}
    \begin{aligned}
         & (\bs{M}_1 - \bs{N}_1)\bs{N}_2 + \bs{M}_1(\bs{M}_2 - \bs{N}_2)               \\
         & = \bs{M}_1\bs{N}_2 - \bs{N}_1\bs{N}_2 + \bs{M}_1\bs{M}_2 - \bs{M}_1\bs{N}_2 \\
         & = \bs{M}_1\bs{M}_2 - \bs{N}_1\bs{N}_2
    \end{aligned}
\end{equation*}
For $p > 2$, we proceed by induction.
Assuming that the statement is true for an arbitrary $p$, we shall prove it holds for $p + 1$.
\begin{equation*}
    \begin{aligned}
         & \bs{M}_1\bs{M}_2\dots \bs{M}_p\bs{M}_{p+1} - \bs{N}_1\bs{N}_2\dots \bs{N}_p\bs{N}_{p+1}                                                      \\
         & = (\underbrace{\bs{M}_1\bs{M}_2\dots \bs{M}_p - \bs{N}_1\bs{N}_2\dots \bs{N}_p}_{\text{HP}})(\bs{N}_{p+1})                                   \\
         & + (\bs{M}_1\bs{M}_2\dots \bs{M}_p)(\bs{M}_{p+1} - \bs{N}_{p+1})                                                                              \\
         & = (\bs{M_1} - \bs{N}_1)\bs{N}_2\dots \bs{N}_p\bs{N}_{p+1} + \bs{M}_1(\bs{M}_2 - \bs{N}_2)\bs{N}_3\dots\bs{N}_p\bs{N}_{p+1}                   \\
         & + \dots + \bs{M}_1\bs{M}_2\dots\bs{M}_{p-1}(\bs{M}_p - \bs{N}_p)\bs{N}_{p+1} + \bs{M}_1\bs{M}_2\dots \bs{M}_p(\bs{M}_{p+1} - \bs{N}_{p+1}) .
    \end{aligned}
\end{equation*}

\subsection{Floating-Point Growth Bounds for Gaussian Elimination}
\label{app:element-wise-bounds}
We are concerned with bounding the elements of the factors of an approximate $\bs{L}\U$ decomposition of a matrix $\A \in \R^{n\times n}$ where partial pivoting is used.
This is commonly expressed in terms of the growth factor~\cite{ref:accuracy-and-stability}, which is defined, for $\A = \bs{L}\U$, as
\begin{equation}
    p_n = \frac{\max_{1 \le i,j,k \le n}{|A^k_{i,j}|}}{\|A_{i,j}\|_{\max}} ,
\end{equation}
where $\A^k$ is the matrix $\A$ after $k - 1$ steps of Gaussian elimination with partial pivoting (therefore $\A^1 = \A$).
The expression
\begin{equation}
    |\U| \le p_n\|A_{i,j}\|_{\max}\bs{1}
\end{equation}
provides the desired upper bounds.
We extend this result to account for floating-point arithmetic.

\subsection{Bounds for the Lower Triangular Factor}
\label{app:element-wise-bounds-l}
We start by considering the lower triangular factor $\bs{L}$.
Since we are using Partial Pivoting, the elements of $|\bs{L}|$ are always bounded by $1$, being obtained by the division of two elements of $\A$ as $p / q : |p| \le |q|$.
Considering the perturbation introduced by the floating-point model (\autoref{sec:floating-point}), the division becomes
\begin{equation}
    \fl{\left|\frac{p}{q}\right|} \le 1 \quad \left|\frac{p}{q}\right| \le 1 + \epsilon .
\end{equation}
Hence, the element-wise bound for $\bs{L}$ is
\begin{equation}
    |\bs{L}| \le (1 + \epsilon) \bs{1} .
\end{equation}

\subsection{Bounds for the Inverse Lower Triangular Factor}
\label{app:element-wise-bounds-l-1}
We have established above that the computed $\bs{L}$ is bounded by $(1 + \epsilon)$.
We also know that the elements on the diagonal of $\bs{L}$ are all $1$.
To maximize the growth of the elements of $\bs{L}^{-1}$, we consider the case where all non-diagonal elements are $-1 - \epsilon$ and simply apply Gaussian Elimination.
We get that
\begin{equation}
    |L_{i,j}^{-1}| \le \begin{cases}
        1                                          & \text{if } i = j   \\
        2^{i - j - 1}(1 + \epsilon)^{3(i - j) - 2} & \text{if } i > j   \\
        0                                          & \text{otherwise} .
    \end{cases}
\end{equation}

\subsection{Bounds for the Upper Triangular Factor}
\label{app:element-wise-bounds-u}
In the general case, when Partial Pivoting is used, the growth factor is $p_n \le 2^{n-1}$~\cite{ref:accuracy-and-stability}.
Given four elements $a, b, p, q$ of $\U^k$, to compute a new entry in $\U^{k+1}$ we perform the following operation: $a - \frac{p}{q}b$.
Recalling that $\left|\frac{p}{q}\right| \le 1$ and taking into account floating-point arithmetic, we conclude that
\begin{equation}
    \fl{a - \frac{p}{q}b} \le (a - \frac{p}{q}b(1 + \epsilon)^2)(1 + \epsilon) \le (|a| + |b|)(1 + \epsilon)^3 \le 2(1 + \epsilon)^3\max(|a|, |b|) .
\end{equation}
Applying this new worst-case multiplier $2(1 + \epsilon)^3$ instead of $2$, we get the bound
\begin{equation}
    |\U| \le (2(1 + \epsilon)^3)^{n-1}\|A_{i,j}\|_{\max}\bs{1} .
\end{equation}

\subsection{Bounds for the Hessenberg Factorization}
\label{app:element-wise-bounds-hessenberg}
If $\bs{H}$ is upper Hessenberg, and Partial Pivoting is used, then the growth factor is $p_n \le n$~\cite{ref:accuracy-and-stability}.
This is because at each iteration we only need to  add a multiple of the pivot row to the next row, possibly swapping these two rows.
Therefore, at the $k$-th step, the $k+1$-th row is still the same as the original $k+1$-th row and the pivot element's modulus is at most $k$ times the largest element of the original matrix.
The same cannot be said for the floating point perturbation, which remain $(1 + \epsilon)^3$ at each step.
Therefore, the bound for performing the Gaussian Elimination over the upper Hessenberg matrix factor is
\begin{equation}
    |\bs{H}| \le n(1 + \epsilon)^{3(n-1)}\|A_{i,j}\|_{\max}\bs{1} .
\end{equation}

\subsection{Backward Error in Solving with the Lower Triangular Factor}
Assuming we are using partial pivoting, the elements of $|\bs{L}|$ are always bounded by $1$, being obtained by the division of two elements of $\B^{(i)}$ as $p / q$ with $|p| \le |q|$.
Considering the perturbation introduced by the floating-point model (\autoref{sec:floating-point}), the division becomes
\begin{equation*}
    \fl{\left|\frac{p}{q}\right|} \le 1, \quad \left|\frac{p}{q}\right| \le 1 + \epsilon .
\end{equation*}
Hence, all elements of $\bs{L}$ are bounded by $1 + \epsilon$, resulting in
\begin{equation}
    \label{eq:bound-l}
    \|\bs{L}\|_1 \le m(1 + \epsilon) .
\end{equation}
From~\cite{ref:accuracy-and-stability} we know that the error introduced in solving the linear system \eqref{eq:linear-systems-error-l} is bounded by
\begin{equation}
    \label{eq:bound-dl}
    \boxed{
        \|\e\bs{L}\|_1 \le \gamma_m \|\bs{L}\|_1 \le \gamma_m m(1 + \epsilon)
    }
\end{equation}
where we recall that $\gamma_m = \frac{m\epsilon}{1 - m\epsilon}$.

\subsection{Backward Error in Applying the Update Factors}
Solving \eqref{eq:linear-systems-error-g} is not done in one step since only the lower triangular matrices $\Gamma^{(i)-1}$ and permutation matrices $\Pi^{i}$ are stored and used directly.
Therefore, to obtain $\bs{z}$ we need to go through a series of matrix multiplications:
\begin{equation*}
    \begin{aligned}
        \bs{z} & = \fl{\G^{(i)-1}\bs{y}} = \fl{\bs{C}^{(1)-1}\bs{C}^{(2)-1}\dots \bs{C}^{(i)-1}\bs{y}}                                                                                                                         \\
               & =\fl{\bs{\Gamma}^{(1)}_{m-1}\bs{\Pi}^{(1)}_{m-1}\dots \bs{\Gamma}^{(1)}_{l_1}\bs{\Pi}^{(1)}_{l_1} \dots  \bs{\Gamma}^{(i)}_{m-1}\bs{\Pi}^{(i)}_{m-1}\dots \bs{\Gamma}^{(i)}_{l_i}\bs{\Pi}^{(i)}_{l_i} \bs{y}} \\
    \end{aligned}
\end{equation*}
where $l_i$ is the index of the leaving variable at the $i$-th iteration.
Consider a single computation on an arbitrary iteration $i$:
\begin{equation*}
    \bs{d} = \fl{\bs{\Gamma}^{(i)}_j\bs{a}}
\end{equation*}
with a generic vector $\bs{a} \in \R^m$, $1 \le j \le m - 1$.
Component-wise, we have
\begin{equation*}
    d_k = \begin{cases}
        a_k                                                                          & \text{if } k \neq j + 1 \\
        \frac{a_{j+1} - g^{(i)}_j a_j (1 + \e^{(i)}_{j,\times})}{1 + \e^{(i)}_{j,-}} & \text{otherwise} ,
    \end{cases}
\end{equation*}
where $\e^{(i)}_{j,\times}$ and $\e^{(i)}_{j,-}$, both bounded by $\epsilon$, are the perturbations introduced by the multiplication and the subtraction, respectively.
Note that, to simplify the following analysis, we assume that $|g^{(i)}_j| \le 1$ exactly and ignore the possible error in their generation, its weight absorbed by $\e^{(i)}_{j,\times}$.
Knowing that $a_k = d_k$ for $k \neq j + 1$, we can write
\begin{equation*}
    a_{j+1} = d_{j+1}(1 + \e^{(i)}_{j,-}) + g^{(i)}_jd_j(1+\e^{(i)}_{j,\times}) .
\end{equation*}
Hence, we can determine the backward error $\e \bs{\Gamma}^{(i)-1}_j$ from the equation $\bs{a} = (\bs{\Gamma}^{(i)-1}_j + \e\bs{\Gamma}^{(i)-1}_j)\bs{d}$ as
\begin{equation*}
    \e\bs{\Gamma}_j^{(i)-1} = \begin{blockarray}{ccccccc}
        1 &   & j & j + 1  && m   \\
        \begin{block}{[cccccc]c}
            0      & \dots  & 0                             & 0              & \dots  & 0      & 1     \\
            \vdots & \ddots &                               &                &        & \vdots &       \\
            0      & \dots  & 0                             & 0              & \dots  & 0      & j     \\
            0      & \dots  & -g^{(i)}_j\e^{(i)}_{j,\times} & \e^{(i)}_{j,-} & \dots  & 0      & j + 1 \\
            \vdots &        &                               &                & \ddots & \vdots &       \\
            0      & \dots  & 0                             & 0              & \dots  & 0      & m     \\
        \end{block}
    \end{blockarray} .
\end{equation*}
Therefore, in order to find $\e\G^{(i)}$ in~\eqref{eq:linear-systems-error-g} we can use the expression
\begin{equation*}
    \begin{aligned}
        \e\G^{(i)} & =  \bs{\Pi}^{(1)}_{l_1}[\bs{\Gamma}^{(1)-1}_{l_1} + \e\bs{\Gamma}^{(1)-1}_{l_1}] \dots \bs{\Pi}^{(1)}_{m-1}[\bs{\Gamma}^{(1)-1}_{m-1} + \e\bs{\Gamma}^{(1)-1}_{m-1}] \dots     \\
                   & \bs{\Pi}^{(i)}_{l_i}[\bs{\Gamma}^{(i)-1}_{l_i} + \e\bs{\Gamma}^{(i)-1}_{l_i}] \dots \bs{\Pi}^{(i)}_{m-1}[\bs{\Gamma}^{(i)-1}_{m-1} + \e\bs{\Gamma}^{(i)-1}_{m-1}] - \G^{(i)} .
    \end{aligned}
\end{equation*}
Note that the error terms do not need to be added to the permutation matrices since no floating-point computation is performed with them.
Introducing
\begin{equation*}
    \tbs{C}^{(i)} = \bs{\Pi}^{(i)}_{l_i}[\bs{\Gamma}^{(i)-1}_{l_i} + \e\bs{\Gamma}^{(i)-1}_{l_i}] \dots \bs{\Pi}^{(i)}_{m-1}[\bs{\Gamma}^{(i)-1}_{m-1} + \e\bs{\Gamma}^{(i)-1}_{m-1}] ,
\end{equation*}
we can write the more compact expression
\begin{equation}
    \label{eq:ctilde-minus-c}
    \e\G^{(i)} = \tbs{C}^{(1)}\tbs{C}^{(2)}\dots \tbs{C}^{(i)} - \bs{C}^{(1)}\bs{C}^{(2)}\dots \bs{C}^{(i)} .
\end{equation}
\new{It is easy to prove (see \autoref{app:matrix-expression}) that the following identity holds for matrices $\bs{M}_1\bs{M}_2\dots \bs{M}_p$ and $\bs{N}_1\bs{N}_2\dots \bs{N}_p$:}
\begin{equation*}
    \begin{aligned}
         & \bs{M}_1\bs{M}_2\dots \bs{M}_p - \bs{N}_1\bs{N}_2\dots \bs{N}_p                                    \\
         & = (\bs{M_1} - \bs{N}_1)\bs{N}_2\dots \bs{N}_p + \bs{M}_1(\bs{M}_2 - \bs{N}_2)\bs{N}_3\dots\bs{N}_p \\
         & + \dots + \bs{M}_1\bs{M}_2\dots\bs{M}_{p-1}(\bs{M}_p - \bs{N}_p) .
    \end{aligned}
\end{equation*}
Thus, to bound $|\e\G^{(i)}|$ we can use the bounds on
\begin{equation*}
    \begin{gathered}
        |\tbs{C}^{(1)}\dots\tbs{C}^{(k)}| \\
        |\tbs{C}^{(k)} - \bs{C}^{(k)}| \\
        |\bs{C}^{(k)}\dots\bs{C}^{(i)}|
    \end{gathered}
\end{equation*}
for each $k = 1, \dots, i$.

Since all permutation matrices are self-inverse, we can rearrange the terms to obtain
\begin{equation*}
    \bs{C}^{(i)} = [\bs{\Pi}^{(i)}_{l_i}\bs{\Pi}^{(i)}_{l_i + 1}\dots\bs{\Pi}^{(1)}_{m-1}]\bs{\Omega}^{(i)}_{l_i}\bs{\Omega}^{(i)}_{l_i + 1}\dots\bs{\Omega}^{(i)}_{m-1}
\end{equation*}
and
\begin{equation*}
    \tbs{C}^{(i)} = [\bs{\Pi}^{(i)}_{l_i}\bs{\Pi}^{(i)}_{l_i + 1}\dots\bs{\Pi}^{(1)}_{m-1}][\bs{\Omega}^{(i)}_{l_i} + \e\bs{\Omega}^{(i)}_{l_i}][\bs{\Omega}^{(i)}_{l_i + 1} + \e\bs{\Omega}^{(i)}_{l_i + 1}]\dots[\bs{\Omega}^{(i)}_{m-1} + \e\bs{\Omega}^{(i)}_{m-1}] ,
\end{equation*}
where
\begin{equation*}
    \begin{aligned}
         & \bs{\Omega}^{(i)}_{m-1} = \bs{\Gamma}^{(i)-1}_{m-1}                                                                                                                                           \\
         & \bs{\Omega}^{(i)}_{m-2} = \bs{\Pi}^{(i)}_{m-1}\bs{\Gamma}^{(i)-1}_{m-2}\bs{\Pi}^{(i)}_{m-1}                                                                                                   \\
         & \bs{\Omega}^{(i)}_{m-3} = \bs{\Pi}^{(i)}_{m-1}\bs{\Pi}^{(i)}_{m-2}\bs{\Gamma}^{(i)-1}_{m-2}\bs{\Pi}^{(i)}_{m-2}\bs{\Pi}^{(i)}_{m-1}                                                           \\
         & \vdots                                                                                                                                                                                        \\
         & \bs{\Omega}^{(i)}_{l_i} = \bs{\Pi}^{(i)}_{m-1}\bs{\Pi}^{(i)}_{m-2}\dots\bs{\Pi}^{(i)}_{l_i + 1}\bs{\Gamma}^{(i)-1}_{l_i}\bs{\Pi}^{(i)}_{l_i + 1}\dots\bs{\Pi}^{(i)}_{m-2}\bs{\Pi}^{(i)}_{m-1} \\
    \end{aligned}
\end{equation*}
and
\begin{equation*}
    \begin{aligned}
         & \bs{\Omega}^{(i)}_{m-1} + \e\bs{\Omega}^{(i)}_{m-1} = \bs{\Gamma}^{(i)-1}_{m-1} + \e\bs{\Gamma}^{(i)-1}_{m-1}                                                                                                                                             \\
         & \bs{\Omega}^{(i)}_{m-2} + \e\bs{\Omega}^{(i)}_{m-2} = \bs{\Pi}^{(i)}_{m-1}[\bs{\Gamma}^{(i)-1}_{m-2} + \e\bs{\Gamma}^{(i)-1}_{m-2}]\bs{\Pi}^{(i)}_{m-1}                                                                                                   \\
         & \bs{\Omega}^{(i)}_{m-3} + \e\bs{\Omega}^{(i)}_{m-3} = \bs{\Pi}^{(i)}_{m-1}\bs{\Pi}^{(i)}_{m-2}[\bs{\Gamma}^{(i)-1}_{m-2} + \e\bs{\Gamma}^{(i)-1}_{m-2}]\bs{\Pi}^{(i)}_{m-2}\bs{\Pi}^{(i)}_{m-1}                                                           \\
         & \vdots                                                                                                                                                                                                                                                    \\
         & \bs{\Omega}^{(i)}_{l_i} + \e\bs{\Omega}^{(i)}_{l_i} = \bs{\Pi}^{(i)}_{m-1}\bs{\Pi}^{(i)}_{m-2}\dots\bs{\Pi}^{(i)}_{l_i + 1}[\bs{\Gamma}^{(i)-1}_{l_i} + \e\bs{\Gamma}^{(i)-1}_{l_i}]\bs{\Pi}^{(i)}_{l_i + 1}\dots\bs{\Pi}^{(i)}_{m-2}\bs{\Pi}^{(i)}_{m-1} \\
    \end{aligned}
\end{equation*}
Note that each $\bs{\Omega}^{(i)}_j$ is a lower triangular matrix with ones on the diagonal and a single non-zero entry somewhere in the sub-diagonal of the $j$-th column whose magnitude is bounded by $1$.
\new{Recalling the definition of $\LI$ from \autoref{sec:preliminaries},} this allows us to determine a bound for their product as
\begin{equation*}
    \left|\prod_{k = l_i}^{m-1}\bs{\Omega}^{(i)}_k \right| \le \LI ,
\end{equation*}
and noticing that $\bs{C}^{(i)}$ is the product of the permutation matrices and the $\bs{\Omega}^{(i)}_j$ matrices, we can write
\begin{equation}
    \label{eq:bound-ci}
    \|\bs{C}^{(i)}\|_1 = \left\| [\bs{\Pi}^{(i)}_{l_i}\bs{\Pi}^{(i)}_{l_i + 1}\dots\bs{\Pi}^{(1)}_{m-1}] \left|\prod_{k = l_i}^{m-1}\bs{\Omega}^{(i)}_k \right|\right\|_1 \le m .
\end{equation}
or, referring to $\G^{(i)}$,
\begin{equation}
    \label{eq:bound-gi}
    \|\G^{(i)}\|_1 \le m^{i} .
\end{equation}
Due to the derivation of $\e\bs{\Omega}^{(i)}_j$ from $\bs{\Gamma}^{(i)-1}_j$, with $1 \le j \le m-1$, we know that
\begin{equation*}
    \e\bs{\Omega}_j^{(i)} = \begin{blockarray}{cccccccc}
        1 &   & j & & \pi_j  && m   \\
        \begin{block}{[ccccccc]c}
            0      & \dots & 0                             & \dots & 0              & \dots  & 0      & 1      \\
            \vdots &       &                               &       &                &        & \vdots &        \\
            0      & \dots & -g^{(i)}_j\e^{(i)}_{j,\times} & \dots & \e^{(i)}_{j,-} & \dots  & 0      & \pi_j \\
            \vdots &       &                               &       &                & \ddots & \vdots &        \\
            0      & \dots & 0                             & \dots & 0              & \dots  & 0      & m      \\
        \end{block}
    \end{blockarray},
\end{equation*}
where the index $\pi_j$, $j < \pi_j \le m$ depends on the permutations applied to $\bs{\Gamma}^{(i)-1}_j$.
Given that $|g^{(i)}_j| \le 1$ and |$\e^{(i)}_{j,-}|, |\e^{(i)}_{j,\times}| \le \epsilon$, we know that all elements of $| \bs{\Omega}^{(i)}_j + \e\bs{\Omega}^{(i)}_j |$ are bounded by $(1 + \epsilon)$.
Therefore, we can bound $\tbs{C}^{(i)}$ with
\begin{equation}
    \label{eq:bound-tci}
    \|\tbs{C}^{(i)}\|_1 \le m(1 + \epsilon)^{m - 1} .
\end{equation}
We use this result to go back to~\eqref{eq:ctilde-minus-c} and find a bound for $|\tbs{C}^{(i)} - \bs{C}^{(i)}|$ by expanding the expression
\begin{equation*}
    \begin{aligned}
         & |[\bs{\Omega}^{(i)}_{l_i} - \e\bs{\Omega}^{(i)}_{l_i}][\bs{\Omega}^{(i)}_{l_i + 1} - \e\bs{\Omega}^{(i)}_{l_i + 1}]\dots[\bs{\Omega}^{(i)}_{m-1} - \e\bs{\Omega}^{(i)}_{m-1}]
        - \bs{\Omega}^{(i)}_{l_i}\bs{\Omega}^{(i)}_{l_i+1}\dots\bs{\Omega}^{(i)}_{m-1}|                                                                                                                                             \\
         & \le |\e\bs{\Omega}^{(i)}_{l_i}\bs{\Omega}^{(i)}_{l_i + 1}\dots\bs{\Omega}^{(i)}_{m-1}| + |\bs{\Omega}^{(i)}_{l_i}\e\bs{\Omega}^{(i)}_{l_i + 1}\dots\bs{\Omega}^{(i)}_{m-1}|                                              \\
         & + \dots + |\bs{\Omega}^{(i)}_{l_i}\bs{\Omega}^{(i)}_{l_i + 1}\dots\e\bs{\Omega}^{(i)}_{m-1}| + |\e\bs{\Omega}^{(i)}_{l_i}\e\bs{\Omega}^{(i)}_{l_i + 1}\dots\bs{\Omega}^{(i)}_{m-1}|                                      \\
         & + \dots + |\bs{\Omega}^{(i)}_{l_i}\bs{\Omega}^{(i)}_{l_i + 1}\dots\e\bs{\Omega}^{(i)}_{m-2}\e\bs{\Omega}^{(i)}_{m-1}| + \dots + |\e\bs{\Omega}^{(i)}_{l_i}\e\bs{\Omega}^{(i)}_{l_i + 1}\dots\e\bs{\Omega}^{(i)}_{m-1}| .
    \end{aligned}
\end{equation*}
Note that the first $m - l_i$ terms in the right-hand side can be bounded by the expression
\begin{equation*}
    |\bs{\Omega}^{(i)}_{l_i}\dots\e\bs{\Omega}^{(i)}_{j}\dots\bs{\Omega}^{(i)}_{m-1}| \le \LI|\e\bs{\Omega}^{(i)}_{j}|\LI \quad \forall j \in \{l_i, \dots, m-1\} .
\end{equation*}
Computing the first multiplication, $\LI|\e\bs{\Omega}^{(i)}_{j}|$, yields a matrix with the structure
\begin{equation*}
    \LI|\e\bs{\Omega}^{(i)}_{j}| = \begin{blockarray}{cccccccc}
        1 &   & j & & \pi_j  && m   \\
        \begin{block}{[ccccccc]c}
            0      & \dots & 0                             & \dots & 0              & \dots  & 0      & 1         \\
            \vdots &       &                               &       &                &        & \vdots &           \\
            0      & \dots & -g^{(i)}_j\e^{(i)}_{j,\times} & \dots & \e^{(i)}_{j,-} & \dots  & 0      & \pi_j     \\
            0      & \dots & -g^{(i)}_j\e^{(i)}_{j,\times} & \dots & \e^{(i)}_{j,-} & \dots  & 0      & \pi_j + 1 \\
            \vdots &       & \vdots                        &       & \vdots         & \ddots & \vdots &           \\
            0      & \dots & -g^{(i)}_j\e^{(i)}_{j,\times} & \dots & \e^{(i)}_{j,-} & \dots  & 0      & m         \\
        \end{block}
    \end{blockarray}.
\end{equation*}
From this, we can determine the element-wise bound
\begin{equation*}
    \LI|\e\bs{\Omega}^{(i)}_{j}|\LI \le (|g_j^{(i)}\e^{(i)}_{j,\times}| + |\e^{(i)}_{j,-}|) \LI \le 2\epsilon\LI, \quad \forall j \in \{l_i, \dots, m-1\} .
\end{equation*}
In a similar fashion, the next $\frac{(m - l_i)(m - l_i - 1)}{2}$ terms are bounded element-wise by
\begin{gather*}
    |\bs{\Omega}^{(i)}_{l_i}\dots\e\bs{\Omega}^{(i)}_{j_1}\dots\e\bs{\Omega}^{(i)}_{j_2}\dots\bs{\Omega}^{(i)}_{m-1}| \le \LI|\e\bs{\Omega}^{(i)}_{j_1}|\LI|\e\bs{\Omega}^{(i)}_{j_2}|\LI \\
    \forall j_1, j_2 : l_i \le j_1 < j_2 \le m-1.
\end{gather*}
Expanding the expression, we notice that $j < \pi_j$ ensures that, in the product of any number of these matrices, column $j$ of all but the last such matrix makes no contribution to the result.
Hence
\begin{gather*}
    \LI|\e\bs{\Omega}^{(i)}_{j_1}|\LI |\e\bs{\Omega}^{(i)}_{j_2}|\LI \le |\e^{(i)}_{j_1,-}|(|g^{(i)}_{j_2}\e^{(i)}_{j_2,\times}| + |\e^{(i)}_{j_2,-}|) \LI \\
    \forall j_1, j_2 : l_i \le j_1 < j_2 \le m-1.
\end{gather*}
We can generalize the above inequality to
\begin{gather*}
    \LI|\e\bs{\Omega}^{(i)}_{j_1}|\LI \dots \LI|\e\bs{\Omega}^{(i)}_{j_N}|\LI \le \left(\prod_{h = 1}^{N - 1}|\e^{(i)}_{j_h,-}|\right) (|g^{(i)}_{j_N}\e^{(i)}_{j_N,\times}| + |\e^{(i)}_{j_N,-}|) \LI \le 2\epsilon^{N} \LI .
\end{gather*}
Therefore, also taking into account the final permutation, the right side of~\eqref{eq:ctilde-minus-c} can be bounded by
\begin{equation}
    \label{eq:bound-dci}
    \begin{aligned}
         & \|\e\bs{C}^{(i)}\|_1 = \|\tbs{C}^{(i)} - \bs{C}^{(i)}\|_1 \le 2m\left[ (m-1)\epsilon + \binom{m-1}{2}\epsilon^2 + \dots + \epsilon^{m-1} \right] \\
         & = 2m((1 + \epsilon)^{m-1} - 1) .
    \end{aligned}
\end{equation}

We can finally put everything together.
Given the results from~\eqref{eq:bound-tci}, \eqref{eq:bound-dci} and \eqref{eq:bound-ci} we can bound the error in~\eqref{eq:linear-systems-error-g} as
\begin{equation*}
    \begin{aligned}
        \|\e\G^{(i)}\|_1 & = \|\tbs{C}^{(1)}\tbs{C}^{(2)}\dots \tbs{C}^{(i)} - \bs{C}^{(1)}\bs{C}^{(2)}\dots \bs{C}^{(i)}\|_1 \\
                         & \le i \|\tbs{C}^{(1)}\tbs{C}^{(2)}\dots \tbs{C}^{(i-1)}\|_1\|\tbs{C}^{(i)} - \bs{C}^{(i)}\|_1      \\
                         & \le i[m(1 + \epsilon)^{m - 1}]^{i-1}  2m((1 + \epsilon)^{m-1} - 1) .                               \\
    \end{aligned}
\end{equation*}
Thus
\begin{equation}
    \label{eq:bound-dgi}
    \boxed{
        \|\e\G^{(i)}\|_1 \leq 2im^i(1 + \epsilon)^{(i-1)(m - 1)}((1 + \epsilon)^{m-1} - 1) .
    }
\end{equation}

\subsection{Bounds for the Inverse Updated Factorization}
Inspired by~\cite{ref:error-growth-bartels-golub}, let us examine the results of the product
\begin{equation*}
    \bs{d} = \fl{|\bs{\Gamma}^{(i)}_j\bs{\Pi}^{(i)}_j\bs{a}|}
\end{equation*}
with a generic $\bs{a} \in \R^m$, $l_i \le j \le m - 1$.
We notice that
\begin{equation*}
    d_k \le \begin{cases}
        |a_k|                                 & \text{if } k \ne j \wedge k \ne j + 1 \\
        \max(|a_k|, |a_{k+1}|)                & \text{if } k = j                      \\
        (|a_k| + |a_{k+1}|)(1 + \epsilon)^{3} & \text{if } k = j + 1 .
    \end{cases}
\end{equation*}
We are arbitrarily making the inequality stronger by using $(1 + \epsilon)^{3}$ instead of $(1 + \epsilon)$.
We can extend this to the sequence of products, $\bs{C}^{(i)-1}$ and perform a similar analysis as
\begin{equation}
    \label{eq:ca}
    \bs{c} = \fl{|\bs{C}^{(i)-1}\bs{a}|} = \fl{|\bs{\Gamma}^{(i)}_{m-1}\bs{\Pi}^{(i)}_{m-1}\dots\bs{\Gamma}^{(i)}_{l_i}\bs{\Pi}^{(i)}_{l_i}\bs{a}|}
\end{equation}
to find that
\begin{equation}
    \label{eq:ca-cases}
    c_k \le \begin{cases}
        \max((1+\epsilon)^{3k}\sum_{v = 1}^{k}|a_v|, |a_{k+1}|) & \text{if } k < m   \\
        (1+\epsilon)^{3m}\sum_{v = 1}^{m}|a_v|                  & \text{if } k = m .
    \end{cases}
\end{equation}
If we then employ the vector norm
\begin{equation}
    \label{eq:norm}
    \|\x\|_\star = \max_{1 \le i \le n}2^{-i}(1 + \epsilon)^{-3i}|x_i| \quad \forall \x \in \R^n ,
\end{equation}
we deduce the condition
\begin{equation*}
    \| \bs{C}^{(i)-1} \bs{a} \|_\star \le 2(1 + \epsilon)^{3} \|\bs{a}\|_\star .
\end{equation*}
See~\autoref{app:star-norm} for proof.
Because $\bs{g}^{(i)} = \bs{G}^{(i)-1}\bs{L}^{-1} = \bs{C}^{(i)-1}\G^{(i - 1)-1}\bs{L}^{-1}$, we can write
\begin{equation}
    \label{eq:gj}
    \|\bs{g^{(i)}_{:,j}}\|_\star \le 2^{i}(1 + \epsilon)^{3i}\|\bs{L}^{-1}_{:,j}\|_\star, \quad 1 \le j \le m .
\end{equation}
Given $\bs{L}$'s values, we can write
\begin{equation*}
    |\bs{L}^{-1}_{k,j}| \le \begin{cases}
        1                                          & \text{if } k = j \\
        2^{k - j - 1}(1 + \epsilon)^{3(k - j) - 2} & \text{if } k > j \\
        0                                          & \text{otherwise}
    \end{cases}
\end{equation*}
which implies the equation
\begin{equation}
    \label{eq:lj-1}
    \|\bs{L}^{-1}_{:,j}\|_\star = 2^{-j}(1 + \epsilon)^{-3j} .
\end{equation}
Finally, together with~\eqref{eq:norm} and~\eqref{eq:gj}, we get
\begin{equation*}
    |g^{(i)}_{k,j}| \le 2^{i + k - j - 1}(1 + \epsilon)^{3(i + k - j) - 2},
\end{equation*}
which we can transform into
\begin{equation}
    \label{eq:bound-gli}
    \boxed{
        \|\bs{G}^{(i)-1}\bs{L}^{-1}\|_1 \le m 2^{m + n - 1}(1 + \epsilon)^{3(m + n) - 2} .
    }
\end{equation}

\subsection{Weighted-Norm Bound for the Update Factors}
\label{app:star-norm}
Employing the vector norm $\|\cdot\|_\star$ \eqref{eq:norm}, we prove that $ \| \bs{T} \bs{a} \|_\star \le 2(1 + \epsilon)^{3}\|\bs{a}\|_\star$, where $|\bs{T}\bs{a}|$ behaves as~\eqref{eq:ca},~\eqref{eq:ca-cases}, by induction over $m$, which is the size of the square matrix $\bs{T}$ and vector $\bs{a}$.
For $m = 1$, we have
\begin{equation}
    \| \bs{T} \bs{a} \|_\star = \frac{1}{2}|a_1| \le (1 + \epsilon)^3|a_1| \le 2(1 + \epsilon)^{3} \|\bs{a}\|_\star .
\end{equation}
Assuming that the statement is true for $m = k$, we will prove it for $m = k + 1$.
We know that $\|[\bs{T}\bs{a}]_{1:k-1}\|_\star \le 2(1 + \epsilon)^{3}\|\bs{a}\|_\star$.
The $k$-th element of $\bs{t} = |\bs{T}\bs{a}|$ is given by
\begin{equation}
    t_k = \max((1 + \epsilon)^{3k}\sum_{i=1}^{k}|a_i|, |a_{k+1}|) .
\end{equation}
If $t_k = |a_{k+1}|$, then we know that
\begin{equation}
    t_{k+1} = (1 + \epsilon)^{3(k + 1)}\sum_{i=1}^{k + 1}|a_i| \le 2(1 + \epsilon)^{3(k + 1)}|a_{k+1}| .
\end{equation}
Since they are at the $k$-th and $k+1$-th position in the vector respectively, their contribution in the norm is bounded by
\begin{equation}
    \begin{aligned}
         & 2^{-k}(1 + \epsilon)^{-3k}|a_{k+1}| \le 2(1 + \epsilon)^3\|\bs{a}\|_\star . \\
    \end{aligned}
\end{equation}
If on the other hand $t_k = (1 + \epsilon)^{3k}\sum_{i=1}^{k}|a_i|$, then the $k$-th element has remained the same, and the last element is
\begin{equation}
    t_{k+1} = (1 + \epsilon)^{3(k + 1)}\sum_{i=1}^{k + 1}|a_i| \le 2(1 + \epsilon)^{3(k + 1)}\sum_{i=1}^{k}|a_i| = 2(1 + \epsilon)^{3}t_{k} ,
\end{equation}
hence the value of $\| \bs{T} \bs{a} \|_\star$ does not change.

\subsection{Backward Error in Solving with the Upper Triangular Factor}
Assuming the factorization utilizes partial pivoting, we can bound the triangular factor of the first iteration by
\begin{equation}
    \label{eq:bound-u0}
    \begin{aligned}
        \|\U^{(0)}\|_1 & \le m(2(1 + \epsilon)^3)^{m-1}\|A\|_{1} .
    \end{aligned}
\end{equation}
See~\autoref{app:element-wise-bounds-u} for more details.
We then check for $i > 0$.
We know that $\U^{(i)}$ is computed by applying the update factors to $\bs{H}^{(i)}$, which corresponds to a Gaussian elimination over an upper Hessenberg matrix.
Therefore, as shown in~\autoref{app:element-wise-bounds-hessenberg}, the values of $\U^{(i)}$ are bounded by the values of $\bs{H}^{(i)}$ as
\begin{equation}
    \label{eq:bound-ui}
    \|\U^{(i)}\|_1 \le m^2(1 + \epsilon)^{3(m-1)}\|\bs{H}^{(i)}\|_{1} .
\end{equation}
Next, we need a bound on $\bs{H}^{(i)}$, which we find with the expression
\begin{equation*}
    \|\bs{H}^{(i)}\|_1 \le \max(\|(\U^{(i-1)}\|_{1}, \|\bs{H}^{(i)}_{:,m}\|_{1})),
\end{equation*}
where
\begin{equation*}
    \bs{H}^{(i)}_{:,m} = \fl{\bs{G}^{(i)-1}\bs{L}^{-1}\A_{:,e_i}},
\end{equation*}
which can be bounded using~\eqref{eq:bound-gli} as
\begin{equation}
    \label{eq:bound-him}
    \|\bs{H}^{(i)}_{:,m}\|_1 = (1 + \epsilon)^{m}\|\G^{(i)-1}\bs{L}^{-1}\|_1\|(\A_{:,e_i})\|_1 .
\end{equation}
The coefficient $(1 + \epsilon)^{m}$ accounts for the computed product of $\G^{(i)-1}\bs{L}^{-1}$ with $\A_{:,e_i}$.
Similar to what we did for $\bs{L}$, we can determine
\begin{equation}
    \label{eq:bound-dui}
    \boxed{
        \|\e\U^{(i)}\|_1 \le \gamma_m  \|\U^{(i)}\|_1
    }
\end{equation}
for~\eqref{eq:linear-systems-error-u}.

\subsection{Backward Error of the Updated Basis}
The first step is to take the initial basis, $\B^{(0)}$, described in~\autoref{sec:preliminaries}, and factorize it into $\bs{L}\U^{(0)}$.
From~\cite{ref:accuracy-and-stability}, we know that
\begin{equation*}
    \B^{(0)} + \e\B^{(0)} = \bs{L}\U^{(0)} \text{ and } |\e\B^{(0)}| \le \gamma_m|\bs{L}||\U^{(0)}| .
\end{equation*}
This produces the bound
\begin{equation}
    \label{eq:bound-db0}
    \|\e\B^{(0)}\|_1 \le \gamma_m 2^{m-1}m^2(1 + \epsilon)^{3m - 2}\|A\|_{1} .
\end{equation}
At the $i$-th iteration, with $i > 0$, we have, column-wise,
\begin{equation*}
    \begin{aligned}
        \B^{(i)} + \e\B^{(i)} & = \begin{bmatrix}
                                      \B^{(i)}_{:,1} & \B^{(i)}_{:,2} & \cdots & \B^{(i)}_{:,m}
                                  \end{bmatrix} + \begin{bmatrix}
                                                      \e\B^{(i)}_{:,1} & \e\B^{(i)}_{:,2} \cdots & \e\B^{(i)}_{:,m} \\
                                                  \end{bmatrix} \\
                              & = \bs{L} \G^{(i)} \begin{bmatrix}
                                                      \U^{(i)}_{:,1} & \U^{(i)}_{:,2} \cdots & \U^{(i)}_{:,m} \\
                                                  \end{bmatrix} .       \\
    \end{aligned}
\end{equation*}
Going to the $(i+1)$-st iteration, we can build $\B^{(i+1)}$ by removing the $l_{i+1}$-th column of $\B^{(i)}$ and appending the column $\A_{e_{i+1}}$.
A similar operation is also performed on $\U^{(i)}$, dropping column $\U^{(i)}_{:,l_{i+1}}$ and appending the column $\bs{H}^{(i+1)}_m$, which is computed as
\begin{equation*}
    \G^{(i)-1}\bs{L}^{-1}\A_{e_{i+1}} = \bs{H}^{(i+1)}_m .
\end{equation*}
We can rearrange the terms to obtain the backward error $\e\A_{e_{i+1}}$ as
\begin{equation*}
    \A_{e_{i+1}} + \e\A_{e_{i+1}} = (\G^{(i)} + \e\G^{(i)})(\bs{L} + \e\bs{L})\bs{H}^{(i+1)}_m ,
\end{equation*}
which implies
\begin{equation*}
    \e\A_{e_{i+1}} = (\bs{L}\e\G^{(i)} + \e\bs{L}\G^{(i)} + \e\bs{L}\e\G^{(i)})\bs{H}^{(i+1)}_m .
\end{equation*}
Therefore, column-wise, we write
\begin{equation*}
    \begin{aligned}
         & \begin{bmatrix}
               \B^{(i)}_{:,1} & \cdots & \B^{(i)}_{:,l_{i-1} - 1} & \B^{(i)}_{:,l_{i+1} + 1} & \cdots & \B^{(i)}_{:,m} & \A_{e_{i+1}} \\
           \end{bmatrix}             \\
         & + \begin{bmatrix}
                 \e\B^{(i)}_{:,1} & \cdots & \e\B^{(i)}_{:,l_{i-1} - 1} & \e\B^{(i)}_{:,l_{i+1} + 1} & \cdots & \e\B^{(i)}_{:,m} & \e\A_{e_{i+1}} \\
             \end{bmatrix} \\
         & = \B^{(i + 1)} + \Delta\B^{(i + 1)}                                                                                                                           \\
         & = \bs{L} \G^{(i)} \begin{bmatrix}
                                 \U^{(i)}_{:,1} & \cdots & \U^{(i)}_{:,l_{i+1} - 1} & \U^{(i)}_{:,l_{i+1} + 1} & \cdots & \U^{(i)}_{:,m} & \bs{H}^{(i+1)}_{:,m} \\
                             \end{bmatrix}              \\
         & = \bs{L} \G^{(i)} \bs{H}^{(i+1)} .
    \end{aligned}
\end{equation*}
We recall that $\bs{C}^{(i + 1)-1}$ is the matrix such that $\bs{C}^{(i + 1)-1}\bs{H}^{(i + 1)} = \U^{(i + 1)}$, and $\G^{(i + 1)} = \G^{(i)}\bs{C}^{(i + 1)}$.
Hence, taking into account the error term, we can write
\begin{equation*}
    \bs{H}^{(i + 1)} = (\bs{C}^{(i+1)} + \e\bs{C}^{(i+1)})\bs{U}^{(i+1)}
\end{equation*}
Therefore,
\begin{equation*}
    \B^{(i + 1)} + \Delta\B^{(i+1)} = \bs{L}\G^{(i+1)}\U^{(i+1)} + \bs{L}\G^{(i + 1)}\e\bs{C}^{(i+1)}\U^{(i+1)} ,
\end{equation*}
and, extracting all the error terms, we find
\begin{equation*}
    \e\B^{(i+1)} = \Delta\B^{(i+1)} - \bs{L}\G^{(i+1)}\e\bs{C}^{(i+1)}\U^{(i+1)} ,
\end{equation*}
which allows us to bound the error $\e\B^{(i+1)}$ in~\eqref{eq:lu-error-b} as
\begin{equation}
    \label{eq:bound-dbi}
    \boxed{
        \begin{aligned}
            \|\e\B^{(i+1)}\|_1 \le \max(\|\e\B^{(i)}\|_1, \|\e\bs{A}_{e_{i+ 1}}\|_1) \\
            +\|\bs{L}\|_1\|\G^{(i+1)}\|_1\|\e\bs{C}^{(i+1)}\|_1\|\U^{(i+1)}\|_1 ,
        \end{aligned}
    }
\end{equation}
with
\begin{equation}
    \begin{aligned}
        \label{eq:bound-dai}
         & \|\e\bs{A}_{e_{i + 1}}\|_1 \le (m \|\e\G^{(i)}\|_1 + m^i\|\e\bs{L}\|_1 + \|\e\bs{G}^{(i)}\|_1\|\e\bs{L}\|_1)\|\bs{H}^{(i+1)}_{:,m}\|_1 \\
         & \|\e\bs{C}^{(i+1)}\|_1\|\U^{(i+1)}\|_1 \le m \|\U^{(i + 1)}\|_1 ,
    \end{aligned}
\end{equation}
assuming $2((1 + \epsilon)^{m} - 1) \le 1$.

\subsection{Proof of~\autoref{lm:errors-linear-systems}}
\begin{proof}
    Expanding~\eqref{eq:define-e} with $\hbs{x}$ and $\hbs{b}$, we obtain
    \begin{equation*}
        \begin{aligned}
             & (\B + \Epsilon)(\x + \e\x) = \bs{b} + \e\bs{b}   \\
             & \B\e\x  + \Epsilon\e\x = \e\bs{b} - \Epsilon\x . \\
        \end{aligned}
    \end{equation*}
    The basis is assumed to be invertible, therefore we can write
    \begin{equation*}
        (\bs{I} + \B^{-1}\Epsilon)\e\x = \B^{-1}(\e\bs{b} - \Epsilon\x) .
    \end{equation*}
    Since $\B^{-1}$ remains constant as $\epsilon \to 0$, the value of $\bs{K} \coloneqq \B^{-1}\Epsilon$ is bounded by
    \begin{equation*}
        \begin{aligned}
            \bs{K} = O(\epsilon) & \qquad & \|\bs{K}\|_1 = O(\epsilon) .
        \end{aligned}
    \end{equation*}
    Thus, if we let $\bs{J} \coloneqq \bs{I} + \bs{K}$ there exists a value $\epsilon_0$ such that, as $\epsilon \to 0$, $\bs{J}$ is diagonally dominant, and hence invertible.
    Therefore, we can write
    \begin{equation}
        \label{eq:extended-linear-system}
        \e\x = \bs{J}^{-1}\B^{-1}(\e\bs{b} - \Epsilon\x) .
    \end{equation}
    Finally, the Neumann series
    \begin{equation*}
        \bs{J}^{-1} = \sum_{n = 0}^{\infty} \bs{K}^{n}
    \end{equation*}
    implies
    \begin{equation*}
        \|\bs{J}^{-1}\|_1 \le \sum_{n = 0}^{\infty}\| \bs{K}^{n} \|_1 = \| \bs{I} \|_1 + \sum_{n = 1}^{\infty}\| \bs{K}^{n} \|_1 = O(1) .
    \end{equation*}
    After expanding, every term in~\eqref{eq:extended-linear-system} is proportional to $O(\epsilon)$.
    Hence, we conclude
    \begin{equation*}
        \|\e\x\|_1 = O(\epsilon) .
    \end{equation*}
    \qed
\end{proof}

\subsection{Proof of~\autoref{lm:errors-reduced-costs}}
\begin{proof}
    We consider the element-wise error
    \begin{equation*}
        \hat{r}_j = \fl{c_j - \hbs{A}^{T}_{:,j}\hbs{y}} \quad 1 \le j \le n .
    \end{equation*}
    From~\cite{ref:accuracy-and-stability}, we know that
    \begin{equation*}
        \begin{aligned}
             & \exists \phi_j : |\phi_j| \le (1 + \epsilon)^{m+1} - 1 :                                                                            \\
             & \hat{r}_j \le (\hat{c}_j - \hbs{A}^{T}_{:,j}\hbs{y})(1 + \phi_j)                                                                    \\
             & r_j + \e r_j \le (c_j + \e c_j)(1 + \phi_j) - (\bs{A}^T_{:,j} + \e\bs{A}^T_{:,j})(\y + \e\y)(1 + \phi_j)                            \\
             & \e r_j \le c_j\phi_j + \e c_j(1 + \phi_j) - \e\bs{A}^T_{:,j}(\y + \e\y)(1 + \phi_j) + \bs{A}^T_{:,j}(\y\phi_j + \e\y(1 + \phi_j)) .
        \end{aligned}
    \end{equation*}
    Every component $\e r_j$ is proportional to $O(\epsilon)$.
    Hence, we conclude
    \begin{equation*}
        \|\e\bs{r}\|_1 = O(\epsilon) .
    \end{equation*}
    \qed
\end{proof}

\subsection{Proof of~\autoref{lm:errors-update}}
\begin{proof}
    Proving the correctness of unboundedness determination in~\autoref{thm:correct-algorithm} (see below) we show that, if $\epsilon$ is small enough, $\hat{d}_j > 0$ if and only if $d_j > 0$.
    For the purpose of determining limiting behavior as $\epsilon \to 0$, we use this property.
    Let $\bs{\eta} \in \R^{m}$ be a vector such that $\|\bs{\eta}\|_{\max} \le \epsilon$.
    Since $\hat{u}_j$ is a simple division, we can express it as
    \begin{equation*}
        \hat{u}_j = \fl{\hat{z}_{\base_j}/\hat{d}_j} = (1 + \eta_j)(\hat{z}_{\base_j}/\hat{d}_j) ,
    \end{equation*}
    from which we can derive
    \begin{equation*}
        \begin{aligned}
            \e u_j = \hat{u}_j - u_j & = (1 + \eta_j)\frac{\zbi[j] + \e \zbi[j]}{d_j + \e d_j} - \frac{\zbi[j]}{d_j}             \\
                                     & = \frac{(1 + \eta_j)(\zbi[j] + \e \zbi[j])d_j - (d_j + \e d_j)\zbi[j]}{(d_j + \e d_j)d_j} \\
                                     & = \frac{(d_j\eta_j - \e d_j)\zbi[j] + (1 + \eta_j)\e\zbi[j]d_j}{(d_j + \e d_j)d_j} .
        \end{aligned}
    \end{equation*}
    Assuming $\epsilon$ is small enough to satisfy the conditions in~\autoref{thm:correct-algorithm}, if $\hat{d}_j > 0$, then the denominator is strictly positive.
    Additionally, let $\dmin$ be the minimum positive value of $d_j$ across all bases and all $1 \le j \le m$.
    From~\autoref{lm:errors-linear-systems} we know that $\|\e \bs{d} \|_1 = O(\epsilon)$.
    It follows that
    \begin{equation*}
        \left\|\frac{1}{(d_j + \e d_j ) d_j}\right\|_1 \le \left\|\frac{1}{((\dmin - \epsilon)\dmin)}\right\|_1 = O(1) \text{ as } \epsilon \to 0.
    \end{equation*}
    This allows us to bound $\|\e \bs{u} \|_1$ with
    \begin{equation}
        \label{eq:bound-eu}
        \| \e \bs{u} \|_1 \le {\|(\eta \cdot \bs{d} - \e\bs{d}) \cdot \zb + (\1_m + \eta) \cdot  \e\zb \cdot \bs{d} \|_1} \left\|\frac{1}{(\dmin - \epsilon)\dmin}\right\|_1 ,
    \end{equation}
    where $\cdot$ indicates the element-wise product of two column vectors.
    Every term in the right-hand side of~\eqref{eq:bound-eu} is proportional to $O(\epsilon)$.
    Hence, we conclude
    \begin{equation*}
        \| \e \bs{u} \|_1  = O(\epsilon) .
    \end{equation*}
    \qed
\end{proof}

\subsection{Proof of~\autoref{thm:correct-algorithm} - Optimality determination and entering variable choice}
\label{pr:correct-algorithm-optimality-determination}
\begin{proof}
    Let $\hbs{r}$ be the reduced cost vector computed at iteration $i$ of the algorithm given tolerance $\tau$ and some precision with unit roundoff $\epsilon$,
    and let $\bs{r}$ be the same value computed by the algorithm with $\tau = 0$ and $\epsilon = 0$.
    Given that $\hbs{r} = \bs{r} + \e\bs{r}$, we want to prove that $r_j + \e r_j + \tau \ge 0$ if and only if $r_j \ge 0$, with $1 \le j \le n$.
    It is clear that the condition $\|\e\bs{r}\|_1 \le \tau$ is sufficient to guarantee that $r_j \ge 0$ implies $r_j + \e r_j + \tau \ge 0$ for all $j$.
    Next, we need to show that $r_j < 0$ implies $r_j + \e r_j + \tau < 0$ for all $j$, which comes naturally if $\|\e\bs{r}\|_1 \le \tau$ and $r_j < 0$ implies $r_j < 2\tau$.

    Since $\bs{r} = \bs{c} - \bs{A}^T\y = \bs{c}- \bs{A}^T\B^{-T}\cb$, $r_j$ is a deterministic function of $\bs{A}, \bs{c}$ and the basis $\B$, of which there are finitely many.
    We can therefore identify a maximum (i.e., least negative) value of $r_j$ across all possible bases, which we call $\rmax$;
    if $r_j$ is always non-negative, we set $\rmax$ to $-\infty$.
    It follows that $-\frac{1}{2}\rmax < \tau$ ensures that $r_j < 0$ implies $r_j < -2\tau$ for all bases and $j$.
    Hence, we need to satisfy $\|\e\bs{r}\|_1 \le \tau < - \frac{1}{2}\rmax$.
    \autoref{lm:errors-reduced-costs} ensures that $\| \e\bs{r} \|_1 = O(\epsilon)$;
    therefore, for any $\tau > 0$ there exists a value $\epsilon$ such that $\|\e\bs{r}\|_1 \le \tau$.
    Taking the minimum of such $\epsilon$'s across all bases makes the statement always hold.
    \qed
\end{proof}

\subsection{Proof of~\autoref{thm:correct-algorithm} - Unboundedness determination}
\label{pr:correct-algorithm-unboundedness-determination}
\begin{proof}
    \label{lm:unboundedness-determination}
    Let $\hbs{d}$ be the entering variable coefficient vector computed at iteration $i$ of the algorithm given tolerance $\tau$ and some precision with unit roundoff $\epsilon$,
    and let $\bs{d}$ be the same value computed by the algorithm with $\tau = 0$ and $\epsilon = 0$.
    Given that $\hbs{d} = \bs{d} + \e\bs{d}$, we want to prove that $d_j + \e d_j - \tau \le 0$ if and only if $d_j \le 0$, with $1 \le j \le m$.
    If $\|\e \bs{d}\|_1 \le \tau$, then $d_i \le 0$ implies $d_i + \e d_i - \tau \le 0$.
    Additionally, if $d_i > 0$ implies $d_i > 2\tau$, then $d_i > 0$ implies $d_i > \tau - \e d_i$, proving the reverse implication.

    Since $\bs{d} = \B^{-1}\A_{:,j}$, $d_j$ is a deterministic function of $\A$ and the basis $\B$, of which there are finitely many.
    We can therefore identify a minimum positive value of $d_j$ across all bases, which we will call $\dmin$;
    if $d_j$ is always non-positive, we set $\dmin$ to $\infty$.
    It follows that $0 < \tau < \frac{1}{2}\dmin$ ensures that $d_j > 0$ implies $d_j > 2\tau$ for all bases and $j$.
    Hence, we need to satisfy $\|\e\bs{d}\|_1 \le \tau < \frac{1}{2}\dmin$.
    \autoref{lm:errors-linear-systems} ensures that $\| \e\bs{d} \|_1 = O(\epsilon)$;
    therefore, for any $\tau > 0$ there exists a value $\epsilon$ such that $\|\e\bs{d}\|_1 \le \tau$.
    Taking the minimum of such $\epsilon$'s across all bases makes the statement always hold.
    \qed
\end{proof}

\subsection{Proof of~\autoref{thm:correct-algorithm} - Leaving variable choice}
\label{pr:correct-algorithm-leaving-variable-choice}
\begin{proof}
    If execution reaches this point, the preceding argument allows us to replace the inequalities $\hat{d}_j > \tau$ with $d_j > 0$ for all $1 \le j \le m$.
    Moreover, we are guaranteed that at least one such value exists, otherwise the procedure would have terminated.

    Let $\hbs{u}$ be the candidate update vector and $\hbs{d}$ the entering variable coefficient computed at iteration $i$ of the algorithm given tolerance $\tau$ and some precision with unit roundoff $\epsilon$,
    and let $\bs{u}$ and $\bs{d}$ be the corresponding values computed by the algorithm with $\tau = 0$ and $\epsilon = 0$.
    Given that $\hbs{u} = \bs{u} + \e\bs{u}$ and $\hbs{d} = \bs{d} + \e\bs{d}$, we want to prove that
    \begin{equation*}
        u_j + \e u_j \le \min(u_j + \e u_j : d_j > 0) + \tau \iff u_j = \min(u_j : d_j > 0) .
    \end{equation*}

    Consider the sequence $\{\alpha_1, \alpha_2, \dots, \alpha_S\}$ of distinct values of $u_j$ where $d_j > 0$, arranged in ascending order.
    Either $S = 1$, in which case all candidates have the minimum value, or $\alpha_2 - \alpha_1 > 0$.
    Let $\beta \coloneqq \infty$ in the first case and $\beta \coloneqq \alpha_2 - \alpha_1$ in the second.
    In either case,
    \begin{equation}
        \label{eq:correct-leaving-variable}
        2\| \e \bs{u} \|_1 \le \tau < \beta - 2 \|\e \bs{u} \|_1
    \end{equation}
    is sufficient to ensure the correct classification of all components.
    This is because, since $\alpha_1 = \min(u_j : d_j > 0)$, we have
    \begin{equation*}
        \alpha_1 - \|\e \bs{u} \|_1 \le \min(u_j + \e u_j : d_j > 0) \le \alpha_1 + \|\e \bs{u} \|_1 .
    \end{equation*}
    Note that, if $u_j = \alpha_1$, then $u_j + \e u_j \le \alpha_1 + \|\e \bs{u} \|_1$.
    Hence, by combining the left-hand sides we get
    \begin{equation*}
        u_j + \e u_j \le \alpha_1 + \| \e \bs{u} \|_1 \le \min(u_j + \e u_j : d_j > 0) + \tau
    \end{equation*}
    correctly classifying the component as $\alpha_1$.
    For the other possibility, suppose $u_j \ge \alpha_2$.
    Noting that $u_j + \e u_j \ge \alpha_2 - \|\e \bs{u} \|_1 = \alpha_1 + \beta - \|\e \bs{u} \|_1$
    and using the right-hand sides we get,
    \begin{equation*}
        \min(u_j + \e u_j : d_j > 0) + \tau < \alpha_1 + \beta - \|\e \bs{u} \|_1 \le u_j + \e u_j ,
    \end{equation*}
    which ensures that the component is not classified as $\alpha_1$.
    Because $u_j = \zbi[j] / d_j = (\B^{-1}\bs{b})_j / (\B^{-1}\bs{A}_{:,e})_j$, where $d_j > 0$ and $u_j = 0$ otherwise,
    we see that $u$, and, by extension, $\beta$, are determined by $\bs{A}, \bs{b}$ and $\B$.

    Let $\umin$ be equal to the smallest positive value of $\beta$ produced in this way across all bases
    with more than one distinct value of $u_j$ where $d_j > 0$ and $\umin = \infty$ if there is only one such value.
    Suppose that we have found some value of $\tau$ satisfying $0 < \tau < \umin$, giving $\tau < \beta$ for every possible $\beta$,
    ensuring that $\umin > 0$.
    Then $\|\e \bs{u} \|_1 < \min(\tau, \umin - \tau)/2$, which is equivalent to $2\| \e \bs{u} \|_1 < \tau < \umin - 2\| \e \bs{u} \|_1$,
    ensures that~\eqref{eq:correct-leaving-variable} holds regardless of the basis.
    Finally, \autoref{lm:errors-update} ensures that, for any given basis, $\| \e \bs{u}\|_1 = O(\epsilon)$ and for any $0 < \tau < \umin$ there exists a value $\epsilon$ such that $\| \e \bs{u} \|_1 < \min(\tau, \umin)/2$.
    Taking the minimum of such $\epsilon$'s across all bases makes the statement always hold.
    \qed
\end{proof}

\subsection{Proof of~\autoref{thm:correct-delta-complete}}
\begin{proof}
    We consider each of the three possible outputs of the algorithm.

    \paragraph{\textbf{Case: Infeasible}}
    If the algorithm returns \underline{infeasible} at~\autoref{line:delta-complete-simplex:infeasible}, then we have found an exactly primal and dual feasible solution to the auxiliary feasibility \gls{lp} problem, defined in lines~\ref{line:delta-complete-simplex:feasibility-lp-start}-\ref{line:delta-complete-simplex:feasibility-lp-end}.
    Since the objective function is the sum of infeasibilities of the original \gls{lp}, and it is minimized, its value being strictly positive implies that the original \gls{lp} is infeasible.

    \paragraph{\textbf{Case: Unbounded}}
    If the algorithm returns \underline{unbounded} at~\autoref{line:delta-complete-simplex:unbounded}
    then there exists an exactly primal feasible basis $\base$ for $\A$ such that the reduced cost vector contains a negative value corresponding to a non-basic variable $e$ that could improve the objective value.
    Since the vector $\B^{-1}\A_{:,e}$ is entirely non-positive, the variable value can be increased arbitrarily without breaking primal feasibility.

    \paragraph{\textbf{Case: $\delta$-optimal}}
    If the algorithm returns \underline{$\delta$-optimal} at~\autoref{line:delta-complete-simplex:delta}, then there exists an exactly primal feasible basic solution $\x$ (primal feasibility being ensured by $\xb = \B^{-1}\bs{b} \ge 0$)
    with primal objective function value $\omega_{U} = \bs{c}^{T}_{\base}\xb$
    and an exactly dual feasible solution $\y$ (dual feasibility being ensured by $\bs{r} = \bs{c} - \A^{T}\y \ge 0$)
    with dual objective function value $\omega_{L} = \bs{b}^T \y$ such that $\omega_{U} - \omega_{L} \le \delta$, which satisfies the definition
    of $\delta$-optimality from~\autoref{def:delta-complete} ($\omega_{U} \ge \omega_{L}$ by weak duality).
    \qed
\end{proof}

\fi

\end{document}